\documentclass{amsart}

\usepackage{amssymb}

\usepackage{amsmath}
\usepackage{amsthm}
\usepackage[usenames]{color}
\usepackage{eepic,epic}
\usepackage{graphicx}
\usepackage{amssymb}
\usepackage{mathtools}
\usepackage{tikz}
\usepackage{psfrag}
\usepackage{xcolor}

\newtheorem{thm}{Theorem}[section]

\newtheorem{lem}[thm]{Lemma}
\newtheorem{prop}[thm]{Proposition}

\theoremstyle{definition}
\newtheorem{defn}[thm]{Definition}
\theoremstyle{remark}

\numberwithin{equation}{section}

\newcommand{\N}{\mathbb N}
\newcommand{\M}{\mathbb M}

\newcommand{\R}{\mathbb R}

\newcommand{\dist}{\operatorname{dist}}

\begin{document}



\title[Adhesion and volume filling in 1-D population dynamics]{Adhesion and volume filling in\\one-dimensional population dynamics\\under Dirichlet boundary condition}

\author{Hyung Jun Choi}
\address{School of Liberal Arts, Korea University of Technology and Education, Cheonan 31253, Republic of Korea}
\curraddr{}
\email{hjchoi@koreatech.ac.kr}
\thanks{}

\author{Seonghak Kim}
\address{Department of Mathematics, College of Natural Sciences, Kyungpook National University, Daegu 41566, Republic of Korea}
\curraddr{}
\email{shkim17@knu.ac.kr}
\thanks{}

\author{Youngwoo Koh}
\address{Department of Mathematics Education, Kongju National University, Kongju 32588, Republic of Korea}
\curraddr{}
\email{ywkoh@kongju.ac.kr}
\thanks{}



\date{}

\dedicatory{}

\begin{abstract}
We generalize the one-dimensional population model of Anguige \& Schmeiser \cite{AS} reflecting the cell-to-cell adhesion and volume filling and classify the resulting equation into the six types. Among these types, we fix one that yields a class of advection-diffusion equations of forward-backward-forward type and prove the existence of infinitely many global-in-time weak solutions to the initial-Dirichlet boundary value problem when the maximum value of an initial population density exceeds a certain threshold. Such solutions are extracted from the method of convex integration by M\"uller \& \v Sver\'ak \cite{MSv2}; they exhibit fine-scale density mixtures over a finite time interval, then become smooth and identical, and decay exponentially and uniformly to zero as time approaches infinity.
\end{abstract}

\maketitle

\tableofcontents

\section{Introduction}

The evolution process for spatial distribution of cells or animals in one-dimensional homogeneous habitat can be modeled by quasilinear advection-diffusion equations of the form,
\begin{equation}\label{eq:first}
u_t=(\rho(u))_{xx},
\end{equation}
where $u=u(x,t)$ denotes the population density of a single species at position $x$ and time $t$, and the derivative $\sigma(s)=\rho'(s)$ of a given flux function $\rho(s)$ represents the diffusivity of equation (\ref{eq:first}).

In zoological studies, as an improvement to the $\Delta$-model by {Taylor \& Taylor} \cite{TT} for adhesive movement, {Turchin} \cite{Tu} proposed equation (\ref{eq:first}), based on a random walk approach \cite{Ok}, as a model of individual movement, which is not only reflecting adhesion or repulsion between conspecific organisms but also avoiding some defects in their model. In his model, the flux function $\rho$ is given by
\begin{equation}\label{rho-function-Turchin}
\rho(s)=\frac{2k_0}{3\omega}s^3 -k_0 s^2 +\frac{\mu}{2} s,
\end{equation}
where $k_0>0$ is the maximum degree of gregariousness, $\omega>0$ is the critical density at which movement switches from adhesive to repulsive, and $\mu\in(0,1]$ is the motility rate. With this flux function $\rho$, he considered the initial-Dirichlet boundary value problem
\begin{equation}\label{ibp-first}
\left\{
\begin{array}{ll}
  u_t=(\rho(u))_{xx} & \mbox{in $\Omega\times(0,\infty)$}, \\
  u=u_0 & \mbox{on $\Omega\times\{t=0\}$}, \\
  u(0,t)=u(L,t)=\delta_0 & \mbox{for $t>0$},
\end{array}\right.
\end{equation}
where $\Omega=(0,L)\subset\R$ is a favorable habitat of size $L>0$, $u_0=u_0(x)$ is the initial density of a single species, and $\delta_0\ge0$ is a constant for the Dirichlet (or absorbing) boundary condition. Here, the constant $\delta_0$ reflects the intensity of hostility of the surrounding habitat and/or the migration rate towards there. For example, when $\delta_0=0$, animals touching the border $\partial\Omega=\{0,L\}$ are permanently lost to the population, either because they move away from the habitat $\Omega$ or because they are killed by predators residing in the very hostile surrounding area. In this case, it is expected that the total population eventually vanishes as time approaches infinity.


In mathematical biology, {Anguige \& Schmeiser} \cite{AS} obtained equation (\ref{eq:first}), based also on the random walk approach, as a model of cell motility which incorporates the effects of cell-to-cell adhesion and volume filling. In their model, the flux function $\rho$ is given by
\begin{equation}\label{rho-function-AS}
\rho(s)=\alpha s^3-2\alpha s^2+s,
\end{equation}
where $\alpha\in [0,1]$ is the adhesion constant. With this flux function $\rho$, they studied the initial-Neumann boundary value problem
\begin{equation}\label{ibp-second}
\left\{
\begin{array}{ll}
  u_t=(\rho(u))_{xx} & \mbox{in $\Omega\times(0,\infty)$}, \\
  u=u_0 & \mbox{on $\Omega\times\{t=0\}$}, \\
  u_x(0,t)=u_x(L,t)=0 & \mbox{for $t>0$},
\end{array}\right.
\end{equation}
where $\Omega=(0,1)\subset\R$ is a habitat of unit size. Here, the Neumann (or reflecting) boundary condition should imply that the total population remains the same for all time.

In both models of \cite{Tu, AS}, when the species in question are highly adhesive (that is, $k_0\omega>\mu$ in (\ref{rho-function-Turchin}) and $\alpha>\frac{3}{4}$ in (\ref{rho-function-AS}), resp.), the diffusivity $\sigma(s)=\rho'(s)$ admits a nonempty open interval in which it is negative so that the equation $u_t=(\rho(u))_{xx}$ becomes \emph{backward parabolic} for the population density values $s=u$ lying in that interval. So when the range of the initial population density $u_0$ overlaps with the interval of backward regime, problems (\ref{ibp-first}) and (\ref{ibp-second}) may be ill-posed.
This awkward situation was regarded as a clue for the authors of \cite{Tu, AS} to expect a certain pattern formation in the population density $u$ as time goes by. However, there have not been any existence results on the highly adhesive random walk models since the usual methods of parabolic theory are no longer applicable to this case.

Regardless of the expectation stated above, there has been skepticism on the random walk models of adhesive population dynamics due to their limited practicality and inability of displaying complicated behavior such as sorting \cite{HPW, CPSZ}. Meanwhile, many nonlocal adhesion models have been suggested and studied actively to reflect various phenomena in population dynamics. For an overview, one may refer to the review article \cite{CPSZ}.

Nonetheless, in this paper, we generalize the adhesion population model of Anguige \& Schmeiser \cite{AS} and classify the resulting equation into the six types. Among such types, we fix one that yields a class of advection-diffusion equations of forward-backward-forward type and prove the existence of infinitely many global-in-time weak solutions to the initial-Dirichlet boundary value problem when the maximum value of the initial population density $u_0$ exceeds a certain threshold. Although our solutions may not exhibit expected behavior of adhesion such as clustering and sorting, they capture some interesting phenomena: fine-scale density mixtures between the high and low density regimes, smoothing after a finite time, and smooth extinction of the species at an exponential rate due to the absorbing boundary condition.
On the other hand, in the same initial value problem under the Neumann boundary condition, we expect to obtain global weak solutions showing quite different behaviors from the Dirichlet case. This will be explored in the subsequent paper as a sequel.

\underline{\textbf{Plan of the paper:}} In the rest of this section, we derive a continuum model of adhesion, which is a generalization of \cite{AS}, and then justify the definition of a global weak solution to the corresponding initial-Dirichlet boundary value problem. In section \ref{sec:classify}, depending on the choice of a pair of adhesion and volume filling constants, we classify the resulting equation from continuum modeling into the six types. Then the main result of the paper, Theorem \ref{thm:main}, is sophisticated in section \ref{sec:mainresult} along with a strategy to obtain solutions. As an independent section, in section \ref{sec:generic}, we formulate a differential inclusion problem and present its special solvability result as Theorem \ref{thm:two-wall}, which serves as an essential ingredient for a proof of Theorem \ref{thm:main} in section \ref{sec:setup}. Sections \ref{sec:proof-main-thm} and \ref{sec:proof-main-lem} carry out a long proof of Theorem \ref{thm:two-wall}, which can be regarded as the core analysis of the paper.

\underline{\textbf{Notations:}} Hereafter, let $L>0$ denote the size of a favorable habitat $\Omega=(0,L)\subset\R$.
For $\tau\in(0,\infty],$ we write
\[
\Omega_\tau:=\Omega\times(0,\tau)\subset\R^2.
\]

Let $k$, $m$, and $n$ be positive integers, let $0<a<1$, and let $U\subset\R^n$ be an open set.
\begin{itemize}
\item[(i)] We denote by $C^k(U)$ the space of functions $u:U\to\R$ whose partial derivatives of order up to $k$ exist and are continuous in $U$.
\item[(ii)] Let $C^k(\bar{U})$ be the space of functions $u\in C^k(U)$ whose partial derivatives of order up to $k$ are uniformly continuous in every bounded subset of $U$.
\item[(iii)] Let $n=2$; so $U\subset\R^2=\R_x\times \R_t.$ We define $C^{2,1}(U)$ to be the space of functions $u:U\to\R$ such that $u_x$, $u_{xx}$, and $u_t$ exist and are continuous in $U$. Also, let $C^{2,1}(\bar{U})$ denote the space of functions $u\in C^{2,1}(U)$ such that $u_x$, $u_{xx}$, and $u_t$ are uniformly continuous in every bounded subset of $U$.
\item[(iv)] We denote by $C^{2+a}(\bar\Omega)$ the space of functions $u\in C^{2}(\bar\Omega)$ with
\[
\sup_{x,y\in\Omega,\,x\ne y}\frac{|u_{xx}(x)-u_{xx}(y)|}{|x-y|^a}<\infty.
\]
\item[(v)] Let $0<T<\infty.$  We define $C^{2+a,1+\frac{a}{2}}(\bar\Omega_T)$ to be the space of functions $u\in C^{2,1}(\bar\Omega_T)$ whose quantities
\[
\sup_{x,y\in\Omega,\,x\ne y,\,0<t<T}\frac{|u_{xx}(x,t)-u_{xx}(y,t)|}{|x-y|^a},\; \sup_{x\in\Omega,\,0<s,t<T,\,s\ne t}\frac{|u_{xx}(x,s)-u_{xx}(x,t)|}{|s-t|^\frac{a}{2}},
\]
\[
\sup_{x,y\in\Omega,\,x\ne y,\,0<t<T}\frac{|u_{t}(x,t)-u_{t}(y,t)|}{|x-y|^a},\; \sup_{x\in\Omega,\,0<s,t<T,\,s\ne t}\frac{|u_{t}(x,s)-u_{t}(x,t)|}{|s-t|^\frac{a}{2}}
\]
are all finite.
\item[(vi)] We denote by $\M^{m\times n}$  the space of $m\times n$ real matrices.
\item[(vii)] For a Lebesgue measurable set $E\subset\R^n,$ its $n$-dimensional measure is denoted by $|E|=|E|_n$ with subscript $n$ omitted if it is clear from the context.
\item[(viii)] We use  $W^{k,p}(U)$ to denote the space of functions $u\in L^{p}(U)$ whose weak partial derivatives of order up to $k$ exist in $U$ and belong to $L^{p}(U)$.
\item[(ix)] A sequence $\{U_\ell\}_{\ell\in\N}$ of disjoint open subsets of $U$ is called a \emph{Vitali cover} of $U$ if $|U\setminus\cup_{\ell\in\N}U_\ell|=0$; in this case, $\{U_\ell\}_{\ell\in\N}$ is said to cover $U$ \emph{in the sense of Vitali}.
\end{itemize}

\subsection{Continuum model}

For any fixed $n\in\N$ with $n\ge 2$, let us consider a discrete distribution of homogeneous cells residing on the endpoints of the $2^n$ uniform subintervals
\[
[0,h_n],\,[h_n,2h_n],\,[2h_n,3h_n],\,\ldots\,,\,[(2^n-1)h_n,L]
\]
of the spatial domain $\bar{\Omega}=[0,L]$, where $h_n:=L/2^{n}$. Assume that $N_{cap}\in\N$ is the maximum number of cells that can stay at each position $x_{ni}:=ih_n$ $(i=0,1,\ldots,2^n)$.  For $i=0,1,\ldots,2^n,$ let $\tilde{u}_{ni}(t)\in\{0,1,\ldots,N_{cap}\}$ denote the number of cells at position $x_{ni}$ and time $t\ge0$, and write
\[
u_{ni}(t):=\frac{\tilde{u}_{ni}(t)}{N_{cap}}\in[0,1],
\]
called the \emph{discrete cell density} at position $x_{ni}$ and time $t\ge0$. We assume that the dynamics of such cells is governed by the system of ordinary differential equations,
\begin{equation}\label{eq:ODE}
\frac{du_{ni}}{dt}=\mathcal{T}^+_{n(i-1)}u_{n(i-1)}+\mathcal{T}^-_{n(i+1)}u_{n(i+1)}-(\mathcal{T}^+_{ni}+\mathcal{T}^-_{ni})u_{ni},
\end{equation}
where $i=2,3,\ldots,2^n-2.$
Here, for $i=1,2,\ldots,2^n-1,$ we denote by  $\mathcal{T}^\pm_{ni}$  the \emph{transitional probabilities} per unit time of a one-step jump from $ih_n$ to $(i\pm 1)h_n$ that are given by
\begin{equation}\label{eq:ODE-tran-prob}
\begin{split}
\mathcal{T}^+_{ni} & =\frac{(1-\beta u_{n(i+1)})(1-\alpha u_{n(i-1)})}{h_n^2},\\
\mathcal{T}^-_{ni} & =\frac{(1-\alpha u_{n(i+1)})(1-\beta u_{n(i-1)})}{h_n^2},
\end{split}
\end{equation}
where $\alpha,\beta\in[0,1]$ are the \emph{adhesion} and \emph{volume filling} constants, respectively.
Note here that only the maximal volume filling constant $\beta=1$ is considered by {Anguige \& Schmeiser} \cite{AS}.
We also write $I_n:=\{ih_n\,|\,i=0,1,\ldots,2^n\}$, the $n$th discrete habitat; then
$
I_2\subset I_3\subset I_4\subset\cdots.
$

Inserting (\ref{eq:ODE-tran-prob}) into (\ref{eq:ODE}), we obtain that for $i=2,3,\ldots,2^n-2$,
\begin{equation}\label{eq:ODE-result}
\begin{split}
\frac{du_{ni}}{dt}
= \frac{1}{h_n}\bigg( & \frac{R_{n(i+1)}-R_{ni}}{h_n}-\frac{R_{ni}-R_{n(i-1)}}{h_n} +\frac{Q_{n(i+1)}-Q_{ni}}{h_n}\\
&-\frac{Q_{ni}-Q_{n(i-1)}}{h_n} +\frac{Q_{ni}-Q_{n(i-1)}}{h_n}-\frac{Q_{n(i-1)}-Q_{n(i-2)}}{h_n}\bigg),
\end{split}
\end{equation}
where
\[
R_{nj}:=u_{nj}(1-\alpha\beta+\alpha\beta(1-u_{n(j-1)})(1-u_{n(j+1)}))
\]
for $j=1,2,\ldots,2^n-1$, and
\[
Q_{nk}:=\alpha(\beta-1) u_{nk}u_{n(k+1)}
\]
for $k=0,1,\ldots,2^n-1$.

To derive the continuum equation, assume that there exists a function $u=u(x,t)\in C^{2,1}(\Omega_\infty;[0,1])\cap C(\bar{\Omega}_\infty;[0,1])$ such that
\[
u(ih_n,t)=u_{ni}(t)
\]
for all $n\in\N,$ $i=0,1,\ldots,2^n,$ and $t\ge 0$.

Now, fix any $n\in\N$ with $n\ge2$. Then for $j=1,2,\ldots,2^n-1$, $k=0,1,\ldots,2^n-1$, and $t\ge 0$, we have
\[
\begin{split}
R_{nj}(t) & =u(jh_n,t)(1-\alpha\beta+\alpha\beta(1-u((j-1)h_n,t))(1-u((j+1)h_n,t))),\\
Q_{nk}(t) & =\alpha(\beta-1) u(kh_n,t)u((k+1)h_n,t).
\end{split}
\]
In this regard,  for $t\ge 0$, define
\[
R_n(x,t)  =u(x,t)(1-\alpha\beta+\alpha\beta(1-u(x-h_n,t))(1-u(x+h_n,t)))
\]
for $h_n\le x\le L-h_n$ and
\[
Q_n(x,t)  =\alpha(\beta-1) u(x,t)u(x+h_n,t)
\]
for $0\le x\le L-h_n;$ then
\[
\begin{split}
R_n & \in C^{2,1}((h_n,L-h_n)\times(0,\infty))\cap C([h_n,L-h_n]\times[0,\infty)),\\
Q_n & \in C^{2,1}((0,L-h_n)\times(0,\infty))\cap C([0,L-h_n]\times[0,\infty)),
\end{split}
\]
\[
\begin{split}
(R_n)_{xx}(x,t)
=& \,u_{xx}(x,t)(1-\alpha\beta+\alpha\beta(1-u(x-h_n,t))(1-u(x+h_n,t)))\\
& -2\alpha\beta u_x(x,t)u_x(x-h_n,t)(1-u(x+h_n,t))\\
& -2\alpha\beta u_x(x,t)u_x(x+h_n,t)(1-u(x-h_n,t))\\
& -\alpha\beta u(x,t)u_{xx}(x-h_n,t)(1-u(x+h_n,t))\\
& +2\alpha\beta u(x,t)u_x(x-h_n,t)u_x(x+h_n,t)\\
& -\alpha\beta u(x,t)u_{xx}(x+h_n,t)(1-u(x-h_n,t)),\\
(Q_n)_{xx}(x,t)
= & \,\alpha(\beta-1) u_{xx}(x,t)u(x+h_n,t) +2\alpha(\beta-1) u_{x}(x,t)u_x(x+h_n,t)\\
& +\alpha(\beta-1) u(x,t)u_{xx}(x+h_n,t).
\end{split}
\]

Next, fix any $x_0\in (\cup_{n\ge2}I_n)\setminus \{0,L\}\subset\Omega$ and $t_0>0.$ Let $n_0$ denote the smallest positive integer with $x_0\in I_{n_0};$ then
\[
x_0=i_0 h_{n_0}
\]
for some odd integer $i_0\in\{1,3,\ldots,2^{n_0}-1\}$. Choose any integer $n>n_0$. Then
\[
x_0=i_{0n} h_{n},
\]
where $i_{0n}:=2^{n-n_0}i_0\in\{2,4,\ldots,2^n-2\}.$ From (\ref{eq:ODE-result}) and the Mean Value Theorem, we have
\[
\begin{split}
u_t(x_0,t_0)= & \,u_t(i_{0n} h_{n},t_0)=\frac{d u_{ni_{0n}}}{dt}(t_0)\\
= & \,(R_n)_{xx}(x_{1n},t_0)+(Q_n)_{xx}(x_{2n},t_0)+(Q_n)_{xx}(x_{3n},t_0)
\end{split}
\]
for some $x_{1n},x_{2n}\in((i_{0n}-1)h_n,(i_{0n}+1)h_n)=(x_0-h_n,x_0+h_n)$ and $x_{3n}\in((i_{0n}-2)h_n,i_{0n}h_n)=(x_0-2h_n,x_0)$.
Thus,
\[
\begin{split}
u_t(x_0,t_0)= & \, u_{xx}(x_{1n},t_0)(1-\alpha\beta+\alpha\beta(1-u(x_{1n}-h_n,t_0))(1-u(x_{1n}+h_n,t_0)))\\
& -2\alpha\beta u_x(x_{1n},t_0)u_x(x_{1n}-h_n,t_0)(1-u(x_{1n}+h_n,t_0))\\
& -2\alpha\beta u_x(x_{1n},t_0)u_x(x_{1n}+h_n,t_0)(1-u(x_{1n}-h_n,t_0))\\
& -\alpha\beta u(x_{1n},t_0)u_{xx}(x_{1n}-h_n,t_0)(1-u(x_{1n}+h_n,t_0))\\
& +2\alpha\beta u(x_{1n},t_0)u_x(x_{1n}-h_n,t_0)u_x(x_{1n}+h_n,t_0)\\
& -\alpha\beta u(x_{1n},t_0)u_{xx}(x_{1n}+h_n,t_0)(1-u(x_{1n}-h_n,t_0))\\
& +\alpha(\beta-1) u_{xx}(x_{2n},t_0)u(x_{2n}+h_n,t_0)\\
& +2\alpha(\beta-1) u_{x}(x_{2n},t_0)u_x(x_{2n}+h_n,t_0)\\
& +\alpha(\beta-1) u(x_{2n},t_0)u_{xx}(x_{2n}+h_n,t_0)\\
& +\alpha(\beta-1) u_{xx}(x_{3n},t_0)u(x_{3n}+h_n,t_0)\\
& +2\alpha(\beta-1) u_{x}(x_{3n},t_0)u_x(x_{3n}+h_n,t_0)\\
& +\alpha(\beta-1) u(x_{3n},t_0)u_{xx}(x_{3n}+h_n,t_0)
\end{split}
\]
\[
\longrightarrow\; 
  u_{xx}(x_0,t_0)(3\alpha\beta u^2(x_0,t_0)-4\alpha u(x_0,t_0)+1) + u_x^2(x_0,t_0)(6\alpha\beta u(x_0,t_0)-4\alpha)
\]
as $n\to\infty;$ that is,
\[
u_t=(\sigma(u)u_x)_x\;\;\mbox{at $(x,t)=(x_0,t_0)$,}
\]
where $\sigma(s)=\sigma_{\alpha\beta}(s):=3\alpha\beta s^2-4\alpha s+1$ $(s\in\R).$ By continuity, we conclude that
\begin{equation}\label{eq:main}
u_t=(\sigma(u)u_x)_x=(\rho(u))_{xx}\;\;\mbox{in $\Omega_\infty$,}
\end{equation}
where $\rho(s)=\rho_{\alpha\beta}(s):=\alpha\beta s^3-2\alpha s^2+s$ $(s\in\R).$
For later use, let us denote
\[
\begin{split}
Z_\sigma=Z_{\sigma,\alpha\beta} & := \{s\in[0,1]\,|\,\sigma(s)=0\}, \\
Z_\rho=Z_{\rho,\alpha\beta} & :=\{s\in(0,1]\,|\,\rho(s)=0\}.
\end{split}
\]

In this paper, we study equation (\ref{eq:main}), coupled with the initial condition,
\begin{equation}\label{incond}
u=u_0\quad\mbox{on $\Omega\times\{t=0\}$}
\end{equation}
and the Dirichlet boundary condition,
\begin{equation}\label{bdry:Dirichlet}
u=0\quad\mbox{on $\partial\Omega\times(0,\infty)$},
\end{equation}
where $u_0\in L^\infty(\Omega;[0,1])$ is a given initial population density.

\subsection{Global weak solutions}
To derive a natural definition of a weak solution to problem (\ref{eq:main})(\ref{incond})(\ref{bdry:Dirichlet}), let $u_0\in C^2(\bar\Omega;[0,1])$ be such that
\[
u_0(0)=u_0(L)=0,\;\;4\alpha (u_0'(0))^2=u_0''(0), \;\;\mbox{and}\;\; 4\alpha (u_0'(L))^2=u_0''(L).
\]

Assume that $u\in C^{2,1}(\bar\Omega_\infty;[0,1])$ is a global classical solution to problem (\ref{eq:main})(\ref{incond})(\ref{bdry:Dirichlet}). Fix any $T>0$, and choose a test function $\varphi\in C^\infty(\bar\Omega\times [0,T])$ such that
\[
\varphi=0\;\;\mbox{on $(\partial\Omega\times[0,T])\cup (\Omega\times\{t=T\})$}.
\]
Then from the integration by parts,
\[
\begin{split}
0= & \int_0^T\int_0^L (u_t -(\rho(u))_{xx})\varphi\,dxdt \\
= & -\int_0^L u_0(x)\varphi(x,0)\,dx - \int_0^T\int_0^L (u\varphi_t +\rho(u)\varphi_{xx})\,dxdt.
\end{split}
\]

Conversely, assume that $u\in C^{2,1}(\bar\Omega_\infty;[0,1])$ is a function satisfying that
\[
\int_0^T\int_0^L (u\varphi_t +\rho(u)\varphi_{xx})\,dxdt+\int_0^L u_0(x)\varphi(x,0)\,dx=0
\]
for each $T>0$ and each $\varphi\in C^\infty(\bar\Omega\times [0,T])$ with
\[
\varphi=0\;\;\mbox{on $(\partial\Omega\times[0,T])\cup (\Omega\times\{t=T\})$}.
\]
We will check below that $u$ is a global classical solution to problem (\ref{eq:main})(\ref{incond})(\ref{bdry:Dirichlet}) provided that
\begin{equation}\label{bdry:Dirichlet-extra}
u\not\in Z_\rho\;\;\mbox{on $\partial\Omega\times(0,\infty)$}.
\end{equation}

To show that (\ref{eq:main}) holds, fix any $\varphi\in C^\infty_c(\Omega_\infty).$ Choose a $T=T_\varphi>0$ so large that $\mathrm{spt}(\varphi)\subset\subset \Omega\times(0,T).$ Then from the integration by parts,
\[
\begin{split}
0= & \int_0^T\int_0^L (u\varphi_t +\rho(u)\varphi_{xx})\,dxdt+\int_0^L u_0(x)\varphi(x,0)\,dx \\
= & \int_0^T\int_0^L (-u_t +(\rho(u))_{xx})\varphi\,dxdt=\int_0^\infty\int_0^L (-u_t +(\rho(u))_{xx})\varphi\,dxdt.
\end{split}
\]
Thus, (\ref{eq:main}) is satisfied.

Next, to check that (\ref{incond}) holds, fix any $\psi\in C^\infty_c(\Omega)$. Choose a function $\omega\in C^\infty(\R)$ such that
\[
\omega=1\;\;\mbox{on $(-\infty,0]$}\;\;\mbox{and}\;\;\omega=0\;\;\mbox{on $[2,\infty)$},
\]
and define $\varphi(x,t)=\psi(x)\omega(t)$ for $(x,t)\in\bar{\Omega}_\infty.$
Then with $T=2,$ it follows from (\ref{eq:main}) that
\[
\begin{split}
0= & \int_0^2\int_0^L (u\varphi_t +\rho(u)\varphi_{xx})\,dxdt+\int_0^L u_0(x)\varphi(x,0)\,dx \\
= & \int_0^L (u(x,2)\varphi(x,2)-u(x,0)\varphi(x,0))\,dx -\int_0^2\int_0^L u_t\varphi\,dxdt \\
& +\int_0^2 (\rho(u(L,t))\varphi_{x}(L,t)- \rho(u(0,t))\varphi_{x}(0,t))\,dt\\
& -\int_0^2 ( (\rho(u))_x (L,t)\varphi(L,t)-(\rho(u))_x (0,t)\varphi(0,t)) \,dt\\
& +\int_0^2\int_0^L (\rho(u))_{xx}\varphi\,dxdt  +\int_0^L u_0(x)\varphi(x,0)\,dx\\
= & \int_0^L (u_0(x)-u(x,0))\psi(x)\,dx.
\end{split}
\]
Thus, (\ref{incond}) is true.

Finally, to see that (\ref{bdry:Dirichlet}) holds, fix any $\omega\in C^\infty_c((0,\infty)).$ Let $T=T_{\omega}>0$ be chosen so large that $\mathrm{spt}(\omega)\subset\subset(0,T)$.
Choose two functions $\psi_0,\psi_1\in C^\infty(\R)$ such that
\[
\psi_0(x)=\left\{ \begin{array}{ll}
            x & \mbox{for $x\le\frac{1}{4}L$} \\[2mm]
            0 & \mbox{for $x\ge\frac{3}{4}L$}
          \end{array}
          \right.
          \;\;\mbox{and}\;\;
\psi_1(x)=\left\{ \begin{array}{ll}
            0 & \mbox{for $x\le\frac{1}{4}L$} \\[2mm]
            -x+L & \mbox{for $x\ge\frac{3}{4}L$}.
          \end{array}
          \right.
\]
For $i=0,1$, define $\varphi_i(x,t)=\psi_i(x)\omega(t)$ for $(x,t)\in\bar\Omega_\infty.$
Then we have from (\ref{eq:main}) that for $i=0,1$,
\[
\begin{split}
0=  & \int_0^L (u(x,T)\varphi_i(x,T)-u(x,0)\varphi_i(x,0))\,dx -\int_0^T\int_0^L u_t\varphi_i\,dxdt \\
& +\int_0^T (\rho(u(L,t))(\varphi_i)_{x}(L,t)- \rho(u(0,t))(\varphi_i)_{x}(0,t))\,dt\\
& -\int_0^T ( (\rho(u))_x (L,t)\varphi_i(L,t)-(\rho(u))_x (0,t)\varphi_i(0,t))\,dt\\
& +\int_0^T\int_0^L (\rho(u))_{xx}\varphi_i\,dxdt  +\int_0^L u_0(x)\varphi_i(x,0)\,dx\\
= & -\int_0^T \rho(u(iL,t))\omega(t)\,dt=-\int_0^\infty \rho(u(iL,t))\omega(t)\,dt;
\end{split}
\]
that is, $\rho(u(x,t))=0$ for all $(x,t)\in\partial\Omega\times(0,\infty)$. Thus, (\ref{bdry:Dirichlet}) follows from (\ref{bdry:Dirichlet-extra}) and the definition of $\rho(s)$.

Summarizing the previous discussion, we have the following.

\begin{prop}
Let $u_0\in C^2(\bar\Omega;[0,1])$ satisfy the compatibility conditions,
\[
u_0(0)=u_0(L)=0,\;\;4\alpha (u_0'(0))^2=u_0''(0), \;\;\mbox{and}\;\; 4\alpha (u_0'(L))^2=u_0''(L).
\]
Assume that $u\in C^{2,1}(\bar\Omega_\infty;[0,1])$ is such that
\begin{equation*}
u\not\in Z_\rho\;\;\mbox{on $\partial\Omega\times(0,\infty)$}.
\end{equation*}
Then $u$ is a global classical solution to problem (\ref{eq:main})(\ref{incond})(\ref{bdry:Dirichlet}) if and only if
\[
\int_0^T\int_0^L (u\varphi_t +\rho(u)\varphi_{xx})\,dxdt+\int_0^L u_0(x)\varphi(x,0)\,dx=0
\]
for each $T>0$ and each $\varphi\in C^\infty(\bar\Omega\times [0,T])$ with
\[
\varphi=0\;\;\mbox{on $(\partial\Omega\times[0,T])\cup (\Omega\times\{t=T\})$}.
\]
\end{prop}

Motivated by this observation, we fix the definition of a global weak solution to problem (\ref{eq:main})(\ref{incond})(\ref{bdry:Dirichlet}) as follows.

\begin{defn}\label{def:global-weak-sol}
Let $u_0\in L^\infty(\Omega;[0,1])$ and $u\in L^\infty(\Omega_\infty;[0,1]).$
\begin{itemize}
\item[(i)] Assume $Z_\rho=\emptyset.$ Then we say that $u$ is a \emph{global weak solution} to problem (\ref{eq:main})(\ref{incond})(\ref{bdry:Dirichlet}) provided that
for each $T>0$ and each $\varphi\in C^\infty(\bar\Omega\times [0,T])$ with
\[
\varphi=0\;\;\mbox{on $(\partial\Omega\times[0,T])\cup (\Omega\times\{t=T\})$},
\]
one has
\begin{equation}\label{eq:weaksol}
\int_0^T\int_0^L (u\varphi_t +\rho(u)\varphi_{xx})\,dxdt+\int_0^L u_0(x)\varphi(x,0)\,dx=0.
\end{equation}
\item[(ii)] Assume $Z_\rho\ne\emptyset$ so that $|Z_\rho|\in\{1,2\}$. Then we say that $u$ is a \emph{global weak solution} to (\ref{eq:main})(\ref{incond})(\ref{bdry:Dirichlet}) provided that
    for each $T>0$ and each $\varphi\in C^\infty(\bar\Omega\times [0,T])$ with
\[
\varphi=0\;\;\mbox{on $(\partial\Omega\times[0,T])\cup (\Omega\times\{t=T\})$},
\]
(\ref{eq:weaksol}) holds and
    that there exists a number $0<\delta<\frac{L}{2}$ with $\min Z_\rho-\delta>0$ such that for each $s_z\in Z_\rho$,
\[
u\not\in[s_z-\delta,s_z+\delta]\;\;\mbox{a.e. in $((0,\delta)\cup(L-\delta,L))\times(0,\infty)$.}
\]
\end{itemize}
\end{defn}

Although
\[
\{(\alpha,\beta)\in[0,1]^2\,|\,Z_{\rho}=Z_{\rho_{\alpha\beta}}\ne\emptyset\}\ne\emptyset,
\]
the adhesion-volume filling pairs $(\alpha,\beta)\in[0,1]^2$ that we mainly consider in this paper satisfy that $Z_{\rho}=Z_{\rho_{\alpha\beta}}=\emptyset$ (see section \ref{sec:mainresult}). Thus, we only have to keep in mind Definition \ref{def:global-weak-sol}(i) even if Definition \ref{def:global-weak-sol}(ii) is included for the sake of completeness.

\section{Classification of equation}\label{sec:classify}
In this section, we classify the types of equation (\ref{eq:main}) as follows. To do so, for any fixed adhesion-volume filling pair $(\alpha,\beta)\in [0,1]^2$, let us write
\[
I^+_{\alpha\beta}:=\{s\in[0,1]\,|\,\sigma(s)>0\}\;\;\mbox{and}\;\; I^-_{\alpha\beta}:=\{s\in[0,1]\,|\,\sigma(s)<0\}.
\]

\underline{\textbf{Case $\beta=1$:}} In this case,
\[
\sigma(s)=3\alpha s^2-4\alpha s+1=3\alpha\Big(s-\frac{2}{3}\Big)^2 +1-\frac{4}{3}\alpha.
\]
\begin{itemize}
\item[(a)] If $0\le \alpha<\frac{3}{4}$, then $\sigma(s)\ge 1-\frac{4}{3}\alpha>0$ for all $s\in[0,1]$; hence, equation (\ref{eq:main}) is \emph{forward parabolic} on $[0,1]$. We may call this case as (F).
\item[(b)] If $\alpha=\frac{3}{4},$ then $\sigma(s)> 1-\frac{4}{3}\alpha=0$ for all $s\in[0,1]\setminus\{\frac{2}{3}\}$ and $\sigma(\frac{2}{3})= 1-\frac{4}{3}\alpha=0$; that is, equation (\ref{eq:main}) is forward parabolic on $[0,1]\setminus\{\frac{2}{3}\}$ and \emph{degenerate} at $s=\frac{2}{3}$. We refer to this case as (FDF).
\item[(c)] If $\frac{3}{4}<\alpha<1,$ then
\[
\sigma(s)\left\{
\begin{array}{ll}
  >0 & \mbox{for $s\in\Big[0,\frac{2\alpha-\sqrt{4\alpha^2-3\alpha}}{3\alpha}\Big)\cup\Big(\frac{2\alpha+\sqrt{4\alpha^2-3\alpha}}{3\alpha},1\Big]=I^+_{\alpha 1}$,} \\[2mm]
  <0 & \mbox{for $s\in\Big(\frac{2\alpha-\sqrt{4\alpha^2-3\alpha}}{3\alpha},\frac{2\alpha+\sqrt{4\alpha^2-3\alpha}}{3\alpha}\Big)=I^-_{\alpha 1}$,} \\[2mm]
  =0 & \mbox{for $s=\frac{2\alpha\pm\sqrt{4\alpha^2-3\alpha}}{3\alpha}$};
\end{array}
\right.
\]
that is, equation (\ref{eq:main}) is forward parabolic on $I^+_{\alpha 1}$, \emph{backward parabolic} on $I^-_{\alpha 1}$, and degenerate at $s=\frac{2\alpha\pm\sqrt{4\alpha^2-3\alpha}}{3\alpha}$. We refer to this case as (FDBDF).
\item[(d)] If $\alpha=1,$ then
\[
\sigma(s)\left\{
\begin{array}{ll}
  >0 & \mbox{for $s\in\big[0,\frac{1}{3}\big)=I^+_{11}$,} \\[2mm]
  <0 & \mbox{for $s\in\big(\frac{1}{3},1\big)=I^-_{11}$,} \\[2mm]
  =0 & \mbox{for $s\in\big\{\frac{1}{3},1\big\}$};
\end{array}
\right.
\]
that is, equation (\ref{eq:main}) is forward parabolic on $I^+_{11}$, backward parabolic on $I^-_{11}$, and degenerate at $s=\frac{1}{3},1$. We refer to this case as (FDBD).
\end{itemize}

\underline{\textbf{Case $\frac{2}{3}<\beta<1$:}} In this case,
\[
\sigma(s)=3\alpha\beta s^2-4\alpha s+1=3\alpha\beta\Big(s-\frac{2}{3\beta}\Big)^2 +1-\frac{4\alpha}{3\beta}.
\]
\begin{itemize}
\item[(a)] If $0\le \alpha<\frac{3}{4}\beta$, then $\sigma(s)\ge 1-\frac{4\alpha}{3\beta}>0$ for all $s\in[0,1]$; hence, equation (\ref{eq:main}) is forward parabolic on $[0,1]$. We refer to this case as (F).
\item[(b)] If $\alpha=\frac{3}{4}\beta,$ then $\sigma(s)> 1-\frac{4\alpha}{3\beta}=0$ for all $s\in[0,1]\setminus\{\frac{2}{3\beta}\}$ and $\sigma(\frac{2}{3\beta})= 1-\frac{4\alpha}{3\beta}=0$; that is, equation (\ref{eq:main}) is forward parabolic on $[0,1]\setminus\{\frac{2}{3\beta}\}$ and \emph{degenerate} at $s=\frac{2}{3\beta}$. We refer to this case as (FDF).
\item[(c)] If $\frac{3}{4}\beta<\alpha<\frac{1}{4-3\beta},$ then
\[
\sigma(s)\left\{
\begin{array}{ll}
  >0 & \mbox{for $s\in\Big[0,\frac{2\alpha-\sqrt{4\alpha^2-3\alpha\beta}}{3\alpha\beta}\Big)\cup\Big(\frac{2\alpha+\sqrt{4\alpha^2-3\alpha\beta}}{3\alpha\beta},1\Big]=I^+_{\alpha \beta}$,} \\[2mm]
  <0 & \mbox{for $s\in\Big(\frac{2\alpha-\sqrt{4\alpha^2-3\alpha\beta}}{3\alpha\beta},\frac{2\alpha+\sqrt{4\alpha^2-3\alpha\beta}}{3\alpha\beta}\Big)=I^-_{\alpha\beta}$,} \\[2mm]
  =0 & \mbox{for $s=\frac{2\alpha\pm\sqrt{4\alpha^2-3\alpha\beta}}{3\alpha\beta}$};
\end{array}
\right.
\]
that is, equation (\ref{eq:main}) is forward parabolic on $I^+_{\alpha\beta}$, backward parabolic on $I^-_{\alpha\beta}$, and degenerate at $s=\frac{2\alpha\pm\sqrt{4\alpha^2-3\alpha\beta}}{3\alpha\beta}$. We refer to this case as (FDBDF).
\item[(d)] If $\alpha=\frac{1}{4-3\beta},$ then
\[
\sigma(s)\left\{
\begin{array}{ll}
  >0 & \mbox{for $s\in\big[0,\frac{1}{4\alpha-1}\big)=I^+_{\alpha\beta}$,} \\[2mm]
  <0 & \mbox{for $s\in\big(\frac{1}{4\alpha-1},1\big)=I^-_{\alpha\beta}$,} \\[2mm]
  =0 & \mbox{for $s\in\big\{\frac{1}{4\alpha-1},1\big\}$};
\end{array}
\right.
\]
that is, equation (\ref{eq:main}) is forward parabolic on $I^+_{\alpha\beta}$, backward parabolic on $I^-_{\alpha\beta}$, and degenerate at $s=\frac{1}{4\alpha-1},1$. We refer to this case as (FDBD).
\item[(e)] If $\frac{1}{4-3\beta}<\alpha\le1,$ then
\[
\sigma(s)\left\{
\begin{array}{ll}
  >0 & \mbox{for $s\in\Big[0,\frac{2\alpha-\sqrt{4\alpha^2-3\alpha\beta}}{3\alpha\beta}\Big)=I^+_{\alpha \beta}$,} \\[2mm]
  <0 & \mbox{for $s\in \Big(\frac{2\alpha-\sqrt{4\alpha^2-3\alpha\beta}}{3\alpha\beta},1\Big]=I^-_{\alpha\beta}$,} \\[2mm]
  =0 & \mbox{for $s=\frac{2\alpha-\sqrt{4\alpha^2-3\alpha\beta}}{3\alpha\beta}$};
\end{array}
\right.
\]
that is, equation (\ref{eq:main}) is forward parabolic on $I^+_{\alpha\beta}$, backward parabolic on $I^-_{\alpha\beta}$, and degenerate at $s=\frac{2\alpha-\sqrt{4\alpha^2-3\alpha\beta}}{3\alpha\beta}$. We refer to this case as (FDB).
\end{itemize}

\underline{\textbf{Case $\beta=\frac{2}{3}$:}} In this case,
\[
\sigma(s)=2\alpha s^2-4\alpha s+1=2\alpha(s-1)^2 +1-2\alpha.
\]
\begin{itemize}
\item[(a)] If $0\le \alpha<\frac{1}{2}$, then $\sigma(s)\ge 1-2\alpha>0$ for all $s\in[0,1]$; hence, equation (\ref{eq:main}) is forward parabolic on $[0,1]$. We refer to this case as (F).
\item[(b)] If $\alpha=\frac{1}{2},$ then $\sigma(s)> 1-2\alpha=0$ for all $s\in[0,1)$ and $\sigma(1)=0$; that is, equation (\ref{eq:main}) is forward parabolic on $[0,1)$ and \emph{degenerate} at $s=1$. We refer to this case as (FD).
\item[(c)] If $\frac{1}{2}<\alpha\le 1,$ then
\[
\sigma(s)\left\{
\begin{array}{ll}
  >0 & \mbox{for $s\in\Big[0,\frac{2\alpha-\sqrt{4\alpha^2-2\alpha}}{2\alpha}\Big)=I^+_{\alpha \frac{2}{3}}$,} \\[2mm]
  <0 & \mbox{for $s\in\Big(\frac{2\alpha-\sqrt{4\alpha^2-2\alpha}}{2\alpha},1\Big]=I^-_{\alpha\frac{2}{3}}$,} \\[2mm]
  =0 & \mbox{for $s=\frac{2\alpha-\sqrt{4\alpha^2-2\alpha}}{2\alpha}$};
\end{array}
\right.
\]
that is, equation (\ref{eq:main}) is forward parabolic on $I^+_{\alpha\frac{2}{3}}$, backward parabolic on $I^-_{\alpha\frac{2}{3}}$, and degenerate at $s=\frac{2\alpha-\sqrt{4\alpha^2-2\alpha}}{2\alpha}$. We refer to this case as (FDB).
\end{itemize}

\underline{\textbf{Case $0<\beta<\frac{2}{3}$:}} In this case,
\[
\sigma(s)=3\alpha\beta s^2-4\alpha s+1=3\alpha\beta\Big(s-\frac{2}{3\beta}\Big)^2 +1-\frac{4\alpha}{3\beta}.
\]
\begin{itemize}
\item[(a)] If $0\le \alpha<\frac{1}{4-3\beta}$, then $\sigma(s)\ge 3\alpha\beta-4\alpha+1>0$ for all $s\in[0,1]$; hence, equation (\ref{eq:main}) is forward parabolic on $[0,1]$. We refer to this case as (F).
\item[(b)] If $\alpha=\frac{1}{4-3\beta},$ then $\sigma(s)> 3\alpha\beta-4\alpha+1=0$ for all $s\in[0,1)$ and $\sigma(1)=0$; that is, equation (\ref{eq:main}) is forward parabolic on $[0,1)$ and degenerate at $s=1$. We refer to this case as (FD).
\item[(c)] If $\frac{1}{4-3\beta}<\alpha\le 1,$ then
\[
\sigma(s)\left\{
\begin{array}{ll}
  >0 & \mbox{for $s\in\Big[0,\frac{2\alpha-\sqrt{4\alpha^2-3\alpha\beta}}{3\alpha\beta}\Big)=I^+_{\alpha \beta}$,} \\[2mm]
  <0 & \mbox{for $s\in\Big(\frac{2\alpha-\sqrt{4\alpha^2-3\alpha\beta}}{3\alpha\beta},1\Big]=I^-_{\alpha \beta}$,} \\[2mm]
  =0 & \mbox{for $s=\frac{2\alpha-\sqrt{4\alpha^2-3\alpha\beta}}{3\alpha\beta};$}
\end{array}
\right.
\]
that is, equation (\ref{eq:main}) is forward parabolic on $I^+_{\alpha\beta}$, backward parabolic on $I^-_{\alpha\beta}$, and degenerate at $s=\frac{2\alpha-\sqrt{4\alpha^2-3\alpha\beta}}{3\alpha\beta}$. We refer to this case as (FDB).
\end{itemize}

\underline{\textbf{Case $\beta=0$:}} In this case,
\[
\sigma(s)=-4\alpha s+1.
\]
\begin{itemize}
\item[(a)] If $0\le \alpha<\frac{1}{4}$, then $\sigma(s)\ge -4\alpha+1>0$ for all $s\in[0,1]$; hence, equation (\ref{eq:main}) is forward parabolic on $[0,1]$. We refer to this case as (F).
\item[(b)] If $\alpha=\frac{1}{4},$ then $\sigma(s)> -4\alpha+1=0$ for all $s\in[0,1)$ and $\sigma(1)=0$; that is, equation (\ref{eq:main}) is forward parabolic on $[0,1)$ and degenerate at $s=1$. We refer to this case as (FD).
\item[(c)] If $\frac{1}{4}<\alpha\le 1,$ then
\[
\sigma(s)\left\{
\begin{array}{ll}
  >0 & \mbox{for $s\in\big[0,\frac{1}{4\alpha}\big)=I^+_{\alpha 0}$,} \\[2mm]
  <0 & \mbox{for $s\in\big(\frac{1}{4\alpha},1\big]=I^-_{\alpha 0}$,} \\[2mm]
  =0 & \mbox{for $s=\frac{1}{4\alpha};$}
\end{array}
\right.
\]
that is, equation (\ref{eq:main}) is forward parabolic on $I^+_{\alpha\beta}$, backward parabolic on $I^-_{\alpha\beta}$, and degenerate at $s=\frac{1}{4\alpha}$. We refer to this case as (FDB).
\end{itemize}

In short, we can summarize the classification of equation (\ref{eq:main}) as in the following figure.

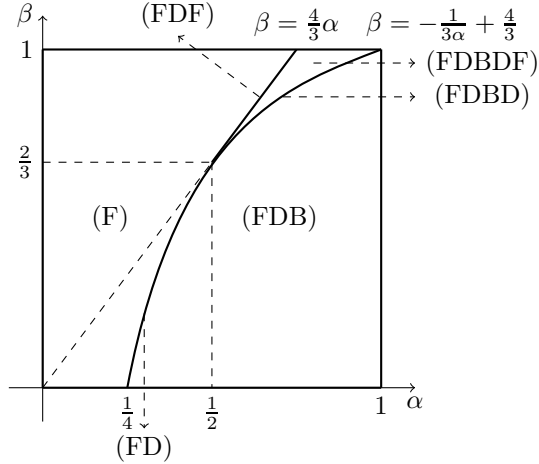
\begin{figure}[ht]
\begin{center}
\begin{tikzpicture}[scale =0.9]
    \draw[->] (-0.5,0) -- (5.5,0);
    \draw[->] (0,-0.5) -- (0,5.5);
    \draw[thick] (0,0) -- (5,0);
    \draw[thick] (5,0) -- (5,5);
    \draw[thick] (5,5) -- (0,5);
    \draw[thick] (0,5) -- (0,0);
    \draw[dashed] (0,0) -- (5/2,10/3);
    \draw[dashed] (0,10/3) -- (5/2,10/3);
    \draw[dashed] (5/2,10/3) -- (5/2,0);
    \draw[thick] (5/2,10/3) -- (15/4,5);
    \draw[thick] (5/4,0) .. controls (2,3.95) and  (3.7,4.5) .. (5,5);
    \draw (0,5.5) node[left] {$\beta$};
    \draw (5.5,0) node[below] {$\alpha$};
    \draw (0,5) node[left] {$1$};
    \draw (5,0) node[below] {$1$};
    \draw (0,10/3) node[left] {$\frac{2}{3}$};
    \draw (5/2,0) node[below] {$\frac{1}{2}$};
    \draw (2.5/2,0) node[below] {$\frac{1}{4}$};
    \draw (15/4,5) node[above] {$\beta=\frac{4}{3}\alpha$};
    \draw (5.9,5) node[above] {$\beta=-\frac{1}{3\alpha}+\frac{4}{3}$};
    \draw (1,5/2) node[] {(F)};
    \draw (3.5,5/2) node[] {(FDB)};
    \draw[dashed,->] (4,4.8) -- (5.5,4.8);
    \draw[dashed,->] (3.55,4.3) -- (5.5,4.3);
    \draw[dashed,->] (3.2,4.3) -- (2,5.2);
    \draw[dashed,->] (1.5,1.1) -- (1.5,-0.6);
    \draw (6.5,4.8) node[] {(FDBDF)};
    \draw (6.5,4.3) node[] {(FDBD)};
    \draw (2,5.5) node[] {(FDF)};
    \draw (1.5,-0.9) node[] {(FD)};

    \end{tikzpicture}
\end{center}
\caption{Classification of equation (\ref{eq:main}).}
\label{fig1}
\end{figure}

\section{Main result}\label{sec:mainresult}
From this section, we mainly study problem (\ref{eq:main})(\ref{incond})(\ref{bdry:Dirichlet}) of type (FDBDF); that is, the initial-Dirichlet boundary value problem in one space dimension,
\begin{equation}\label{ib-P}
\begin{cases}
u_{t} =(\sigma(u)u_x)_x=(\rho(u))_{xx} & \mbox{in $\Omega_\infty$,}\\
u =u_0 & \mbox{on $\Omega\times \{t=0\}$},\\
u=0 & \mbox{on $\partial\Omega\times(0,\infty)$,}
\end{cases}
\end{equation}
where $u_0:\Omega\to [0,1]$ is a given initial population density, $u(x,t)\in[0,1]$ represents the population density at a space-time point $(x,t)\in\Omega_\infty$, the diffusivity $\sigma:\R\to\R$ is given by
\[
\sigma(s)=3\alpha\beta s^2- 4\alpha s +1\;\;(s\in\R)
\]
for some constants $\frac{2}{3}<\beta\le 1$ and $\frac{3}{4}\beta<\alpha<\frac{1}{4-3\beta}$, and
\[
\rho(s)=\alpha\beta s^3-2\alpha s^2+s \;\;(s\in\R).
\]

Letting
\[
s^\pm_{0}=\frac{2\alpha \pm \sqrt{4\alpha^2 -3\alpha\beta}}{3\alpha\beta},
\]
we observe that $0<s^-_0<s^+_0<1$ and that
\begin{equation}\label{sigma-1}
\sigma(s) \left\{ \begin{array}{ll}
                    >0 & \mbox{for $s\in I^+_{\alpha\beta}=[0,s^-_0)\cup(s^+_0,1]$,} \\[1mm]
                    <0 & \mbox{for $s\in I^-_{\alpha\beta}=(s^-_0,s^+_0)$,} \\[1mm]
                    =0 & \mbox{for $s\in Z_{\sigma}=\{s^+_0,s^-_0\}$.}
                  \end{array}
\right.
\end{equation}
We check below that
\begin{equation}\label{ineq-rho}
\rho(s)>0\;\;\forall s\in(0,1];
\end{equation}
hence, our global weak solutions to problem (\ref{ib-P}) should be as in Definition \ref{def:global-weak-sol}(i).
Note from $\frac{2}{3}<\beta\le 1$ that
\[
3\beta^2-4\beta+1=(3\beta-1)(\beta-1)\le 0
\]
so that
\[
\alpha<\frac{1}{4-3\beta}\le\beta;
\]
thus, $\alpha^2-\alpha\beta=\alpha(\alpha-\beta)<0$. This implies that the equation $\rho(s)=s(\alpha\beta s^2-2\alpha s +1)=0$ has precisely one zero $s=0$ in $\R$; hence, inequality (\ref{ineq-rho}) holds.

Let
\[
r^*=\min\{\rho(s^-_0),\rho(1)\};
\]
then from (\ref{sigma-1}) and (\ref{ineq-rho}), we have
\[
r^*>\rho(s^+_0)>0.
\]
For each $r\in[\rho(s^+_0),r^*],$ let $s^+(r)\in[s^+_0,1]$ and $s^-(r)\in(0,s^-_0]$ denote the unique numbers with
\[
\rho(s^\pm(r))=r.
\]
Let us write
\[
s^\pm_1:=s^\pm(\rho(s^+_0))\;\;\mbox{and}\;\; s^\pm_2:=s^\pm(r^*);
\]
then
\[
0<s^-_1<s^-_2\le s^-_0<s^+_0 =s^+_1<s^+_2\le 1,
\]
and
\[
s^-_2= s^-_0\;\;\mbox{or}\;\; s^+_2= 1
\]
(see Figure \ref{fig2}).

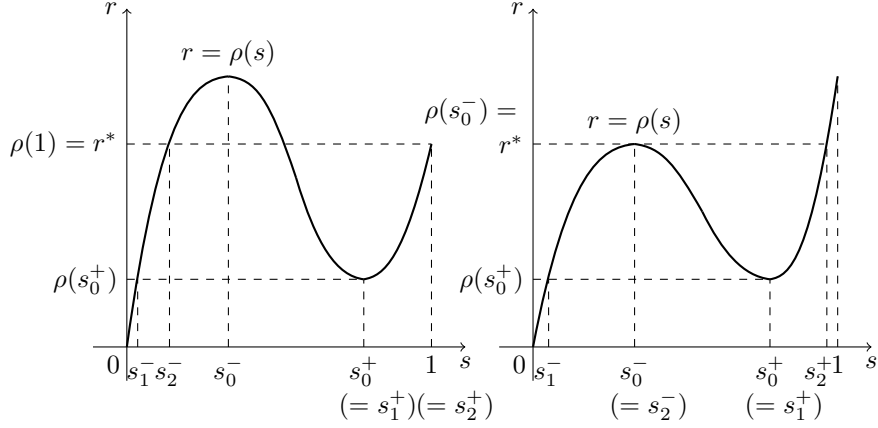
\begin{figure}[ht]
\begin{center}
\begin{tikzpicture}[scale =0.9]
    \draw[->] (-0.5,0) -- (5,0);
    \draw[->] (5.5,0) -- (11,0);
    \draw[->] (0,-0.5) -- (0,5);
    \draw[->] (6,-0.5) -- (6,5);
 \draw[dashed] (0,1)--(3.5,1);
 \draw[dashed] (0,3)--(4.5,3);
 \draw[dashed] (1.5,0)--(1.5,4);
 \draw[dashed] (3.5,0)--(3.5,1);
 \draw[dashed] (4.5,0)--(4.5,3);
 \draw[dashed] (0.16,0)--(0.16,1);
 \draw[dashed] (0.63,0)--(0.63,3);
 \draw (1.5, 0) node[below] {$s^-_0$};
 \draw (3.5, 0) node[below] {$s^+_0$};
 \draw (3.71, -0.5) node[below] {$(=s^+_1)$};
 \draw (4.85, -0.5) node[below] {$(=s^+_2)$};
 \draw (0, 3) node[left] {$\rho(1)=r^*$};
 \draw (0, 1) node[left] {$\rho(s^+_0)$};
 \draw (0.22, 0) node[below] {$s^-_1$};
 \draw (0.63, 0) node[below] {$s^-_2$};
 \draw[dashed] (6,1)--(9.5,1);
 \draw[dashed] (6,3)--(10.3,3);
 \draw[dashed] (7.5,0)--(7.5,3);
 \draw[dashed] (9.5,0)--(9.5,1);
 \draw[dashed] (10.5,0)--(10.5,4);
 \draw[dashed] (6.23,0)--(6.23,1);
 \draw[dashed] (10.34,0)--(10.34,3);
 \draw (7.5, 0) node[below] {$s^-_0$};
 \draw (9.5, 0) node[below] {$s^+_0$};
 \draw (9.71, -0.5) node[below] {$(=s^+_1)$};
 \draw (7.71, -0.5) node[below] {$(=s^-_2)$};
 \draw (6, 3) node[left] {$r^*$};
 \draw (5.9, 3.5) node[left] {$\rho(s^-_0)=$};
 \draw (6, 1) node[left] {$\rho(s^+_0)$};
 \draw (6.23, 0) node[below] {$s^-_1$};
	\draw[thick]   (0, 0) .. controls (0.5,3.6) and  (1,3.95)   ..(1.5,4);
    \draw[thick]   (1.5,4) .. controls (1.9,3.95) and  (2.1,3.7)   ..(2.5,2.5);
    \draw[thick]   (2.5,2.5) .. controls (2.8,1.5) and  (3.1,1.05)   ..(3.5,1);
	\draw[thick]   (3.5,1) .. controls  (3.8,1.05) and (4.1,1.3) ..(4.5,3);
    \draw[thick]   (6, 0) .. controls (6.5,2.6) and  (7,2.95)   ..(7.5,3);
    \draw[thick]   (7.5,3) .. controls (7.9,2.95) and  (8.1,2.7)   ..(8.5,2);
    \draw[thick]   (8.5,2) .. controls (8.8,1.4) and  (9.1,1.05)   ..(9.5,1);
	\draw[thick]   (9.5,1) .. controls  (9.8,1.05) and (10.1,1.3) ..(10.5,4);
	\draw (11,0) node[below] {$s$};
    \draw (5,0) node[below] {$s$};
    \draw (1.5, 4) node[above] {$r=\rho(s)$};
    \draw (7.5, 3) node[above] {$r=\rho(s)$};
   \draw (4.5, 0) node[below] {$1$};
   \draw (10.5, 0) node[below] {$1$};
   \draw (10.22, 0) node[below] {$s^+_2$};
   \draw (0, 5) node[left] {$r$};
   \draw (6, 5) node[left] {$r$};
   \draw (-0.2, 0) node[below] {$0$};
   \draw (5.8, 0) node[below] {$0$};
    \end{tikzpicture}
\end{center}
\caption{Two possible graphs of  $r=\rho(s)$ $(0\le s\le 1)$.}
\label{fig2}
\end{figure}



\underline{\textbf{Initial population density:}} We assume that the initial population density $u_0$ to problem (\ref{ib-P}) fulfills the regularity condition,
\[
u_0\in C^{2+a}(\bar\Omega;[0,1])\;\;\mbox{for some $0<a<1$}
\]
and the compatibility conditions,
\[
u_0(0)=u_0(L)=0,\;\; 4\alpha(u_0'(0))^2=u_0''(0),\;\;\mbox{and}\;\; 4\alpha(u_0'(L))^2=u_0''(L).
\]
We write
\[
M_0:=\max_{\bar{\Omega}}u_0 \in[0,1].
\]

From the classical parabolic theory, we can handle the following case that the maximum value $M_0$ of the initial population density $u_0$ is strictly less than the infimum $s^-_0$ of the backward regime $(s^-_0,s^+_0)$.

\begin{thm}[Case $M_0<s^-_0$: smooth extinction]\label{thm:classical} Assume $M_0<s^-_0.$ Then there exists a unique global classical solution $u\in C^{2,1}(\bar\Omega_\infty;[0,1])$ to problem (\ref{ib-P}) such that
\[
u\in C^{2+a,1+\frac{a}{2}}(\bar\Omega_T;[0,1])
\]
for all $T>0$ and that $0\le u\le M_0$ in $\Omega_\infty$.
Moreover, it follows that
\[
0\le u(x,t)\le \max_{\bar{\Omega}}u(\cdot,s)
\]
for all $x\in\Omega$ and $0\le s\le t$ and that there exist two constants $C\ge0$ and $\gamma>0$, depending only on $M_0$, $L$, $\alpha$, and $\beta$, such that
\[
\|u(\cdot,t)\|_{L^\infty(\Omega)}\le C e^{-\gamma t}
\]
for all $t\ge0$.
\end{thm}

\begin{proof}
Let $M_1=\frac{M_0+s^-_0}{2}.$
By elementary calculus, we can choose a function $\rho^\star\in C^3(\R)$ such that
\[
\rho^\star=\rho\;\;\mbox{on $(-\infty,M_1]$}\;\;\mbox{and}\;\;(\rho^\star)'\ge c_0\;\;\mbox{in $\R$,}
\]
for some constant $c_0>0.$
Then from \cite[Theorem 12.14]{Ln}, the modified problem,
\begin{equation*}
\begin{cases}
u^\star_{t} =(\rho^\star(u^\star))_{xx} & \mbox{in $\Omega_\infty$,}\\
u^\star =u_0 & \mbox{on $\Omega\times \{t=0\}$},\\
u^\star=0 & \mbox{on $\partial\Omega\times(0,\infty)$,}
\end{cases}
\end{equation*}
admits a unique global classical solution $u^\star\in C^{2,1}(\bar\Omega_\infty;[0,1])$ with
\[
u^\star\in C^{2+a,1+\frac{a}{2}}(\bar\Omega_T;[0,1])\;\;\forall T>0.
\]
Also, it follows from  \cite[Theorem 1.1]{Ki} that
\begin{equation*}
0 \le u^\star(x,t)\le\|u^\star(\cdot,s)\|_{L^\infty(\Omega)}\;\;\mbox{for all $x\in\bar\Omega$ and $t\ge s\ge0$}
\end{equation*}
and that
\begin{equation*}
\|u^\star(\cdot,t)\|_{L^\infty(\Omega)}\le C e^{-\gamma t}\;\;\mbox{for all $t\ge0$},
\end{equation*}
for some constants $C\ge 0$ and $\gamma>0$, depending only on $M_0$, $L$, $\alpha$, and $\beta$. In particular, $0\le u^\star\le M_0<M_1$ in $\Omega_\infty;$
thus, from the choice of $\rho^\star,$ we see that $u:=u^\star$ is a global classical solution to problem (\ref{ib-P}).

Suppose $\tilde{u}\in C^{2,1}(\bar\Omega_\infty;[0,1])$ is a unique global classical solution to problem (\ref{ib-P}) such that
\[
\tilde{u}\in C^{2+a,1+\frac{a}{2}}(\bar\Omega_T;[0,1])
\]
for all $T>0$ and that $0\le \tilde{u}\le M_0$ in $\Omega_\infty$. Then from the choice of $\rho^\star,$ $\tilde{u}$ is also a global classical solution to the modified problem above. By uniqueness of the modified problem, we have $\tilde{u}=u^\star=u$ in $\Omega_\infty$.
\end{proof}

As the main result of the paper, we present the following theorem on the case that the maximum value $M_0$ of the initial population density $u_0$ exceeds the threshold $s^-_1$. A proof of this theorem is given in section \ref{sec:setup}.

\begin{thm}[Case $M_0>s^-_1$: density mixtures and smooth extinction]\label{thm:main}
Assume $M_0>s^-_1.$
Let $r_1\in[\rho(s^+_0),r^*)$ be any number such that
\[
s^-(r_1)<M_0,
\]
and let $r_2\in(r_1,r^*]$.
Then there exist a function $u^\star\in C^{2,1}(\bar\Omega_\infty;[0,1])$ with
\[
u^\star\in C^{2+a,1+\frac{a}{2}}(\bar\Omega_T;[0,1])\;\;\forall T>0,
\]
a nonempty bounded open set $Q\subset\Omega_\infty$ with
\[
\bar{Q}\subset \Omega\times[0,\infty)\;\;\mbox{and}\;\;\bar{Q}\cap (\Omega\times\{0\})\ne\emptyset,
\]
and infinitely many global weak solutions $u\in L^{\infty}(\Omega_\infty;[0,1])$ to problem (\ref{ib-P}) satisfying the following:
\begin{itemize}
\item[(a)] \underline{\emph{Smoothing in finite time:}}
\[
u=u^\star\;\;\mbox{in $\Omega_\infty\setminus \bar{Q};$}
\]
\item[(b)] \underline{\emph{Density mixtures:}}
\[
u\in[s^-(r_1),s^-(r_2)]\cup[s^+(r_1),s^+(r_2)]\;\;\mbox{a.e. in $Q;$}
\]
\item[(c)] \underline{\emph{Fine-scale oscillations:}} for any nonempty open set $O\subset Q,$
\[
\underset{O}{\mathrm{ess\,osc}}\, u:=\underset{O}{\mathrm{ess\,sup}}\, u-\underset{O}{\mathrm{ess\,inf}}\, u \ge s^+(r_1)-s^-(r_2)>0;
\]
\item[(d)] \underline{\emph{Nonincreasing total population:}}
\[
\int_\Omega u(x,t)\,dx= \int_\Omega u^\star(x,t)\,dx\;\;\forall t\ge 0,
\]
and the function $t\mapsto \int_\Omega u^\star(x,t)\,dx$ is nonincreasing on $[0,\infty);$
\item[(e)] \underline{\emph{Maximum principle:}}
\[
0\le u^\star(x,t)\le\|u^\star(\cdot,s)\|_{L^\infty(\Omega)}\le 1\;\;\forall x\in\bar\Omega,\,\forall t\ge s\ge0;
\]
\item[(f)] \underline{\emph{Exponential and smooth extinction:}} there exists two constants $C>0$ and $\gamma>0$ such that
\[
\|u^\star(\cdot,t)\|_{L^\infty(\Omega)}\le C e^{-\gamma t}\;\;\forall t\ge0.
\]
\end{itemize}
\end{thm}

Since $Q\ne \emptyset$ is a bounded subset of $\Omega_\infty$, it follows from (a) that the solutions $u$ are smooth and identical on $\bar\Omega\times [t^*,\infty)$, where
\[
0<t^*:=\sup_{(x,t)\in Q} t<\infty.
\]
So properties (e) and (f) are valid for the solutions $u$ when $t\ge t^*.$
Since $\bar{Q}\subset \Omega\times[0,\infty)$, observe from (a) that the solutions $u$ are smooth and identical near the cylindrical boundary $\partial\Omega\times[0,\infty)$; thus, $u=0$ pointwise on $\partial\Omega\times[0,\infty)$. From (b), (c), and $\bar{Q}\cap (\Omega\times\{0\})\ne\emptyset,$ the solutions $u$ experience immediate fine-scale density mixtures in $Q$ between low density regime $[s^-(r_1),s^-(r_2)]$ and high density regime $[s^+(r_1),s^+(r_2)]$. Observe also from (d), (e), and (f) that the total population in $\Omega$ decreases exponentially in time:
\[
0\le \int_\Omega u(x,t)\,dx= \int_\Omega u^\star(x,t)\,dx\le CLe^{-\gamma t}\;\;\forall t\ge0.
\]

Instead of the initial-Dirichlet boundary value problem (\ref{ib-P}), one may consider the Cauchy problem with the  diffusivity $\sigma:\R\to\R$ as in (\ref{ib-P}) and initial population density $u_0\in C^{2+a}_c(\R;[0,1]).$ In this case, it is essential to study first the existence and properties of a global classical solution to the modified Cauchy problem as in (\ref{ib-modi}) with $\Omega$ replaced by $\R$. On the other hand, in a biological viewpoint, it is also interesting to consider in problem (\ref{ib-P}) a Dirac delta distribution as the initial population density $u_0$. This may be regarded as the case of describing population dynamics of a species, concentrated initially at a single point. Likewise, one then needs to establish first the existence and properties of a global classical solution to the modified (initial-boundary value or Cauchy) problem under the initial Dirac delta distribution $u_0$.

\underline{\textbf{Approach by differential inclusion:}} Let us take a moment here to explain our approach to prove Theorem \ref{thm:main}.
To solve the equation in (\ref{ib-P}), we formally put $v_x=u$ in $\Omega_\infty$ for some function $v:\Omega_\infty\to\R$; so we consider the equation,
\[
v_t=(\rho(v_x))_x\;\;\mbox{in $\Omega_\infty$.}
\]

To solve the previous equation in the sense of distributions in $\Omega_\infty,$
we may try to find a vector function $z=(v,w)\in W^{1,\infty}(\Omega_\infty;\R^{1+1})$ with $v_x\in L^\infty(\Omega_\infty;[0,1])$ such that
\begin{equation}\label{alt-system}
w_x=v\;\;\mbox{and}\;\;w_t=\rho(v_x)\;\;\mbox{a.e. in $\Omega_\infty$}.
\end{equation}
If there is such a function $z=(v,w),$ we take $u=v_x \in L^\infty(\Omega_\infty;[0,1])$; then from the integration by parts, for each $\varphi\in C^\infty_c(\Omega_\infty)$,
\[
\begin{split}
\int_0^\infty\int_0^L  & (u\varphi_t +\rho(u)\varphi_{xx})\,dxdt = \int_0^{T_\varphi}\int_0^L (v_x\varphi_t+\rho(v_x)\varphi_{xx})\,dxdt \\
 = & \int_0^{T_\varphi}\int_0^L (-v\varphi_{tx}+w_t\varphi_{xx})\,dxdt = \int_0^{T_\varphi}\int_0^L (-v\varphi_{tx}+w_x\varphi_{xt})\,dxdt=0,
\end{split}
\]
where $T_\varphi:=\sup_{(x,t)\in\mathrm{spt}(\varphi)}t +1$. Hence, $u$ is a global weak solution of the equation in (\ref{ib-P}) in the sense of distributions in $\Omega_\infty.$

On the other hand, for each $b\in\R,$ define
\[
\Sigma(b)=\Sigma_{\rho}(b)=\left\{
\begin{pmatrix}
s & c \\
b & \rho(s)
\end{pmatrix}\in\M^{2\times 2}\,\Big|\, 0\le s\le 1,\, c\in\R
\right\};
\]
then system (\ref{alt-system}) is equivalent to the inhomogeneous partial differential inclusion,
\[
\nabla z=\begin{pmatrix}
v_x & v_t \\
w_x & w_t
\end{pmatrix}\in \Sigma(v)\;\;\mbox{a.e. in $\Omega_\infty$},
\]
where $\nabla=(\partial_x,\partial_t)$ is the space-time gradient operator. In this regard, utilizing the method of convex integration by M\"uller \& \v Sver\'ak \cite{MSv2}, we aim at solving this inclusion for certain sets $K(b)\subset\Sigma(b)$ $(b\in\R)$ in a generic setup (section \ref{sec:generic}) while reflecting the initial and Dirichlet boundary conditions in (\ref{ib-P}).

After the successful understanding of homogeneous partial differential inclusions in the study of crystal microstructures by Ball \& James \cite{BJ} and Chipot \& Kinderlehrer \cite{CK}, the methods of convex integration in differential inclusions have been extensively applied to many important problems; see, e.g., elliptic systems \cite{MSv2}, the Euler equations and Onsager's conjecture \cite{DS, Is}, the porous media equation \cite{CFG}, active scalar equations \cite{Sy}, and the Muskat problem \cite{CCF}.

\section{Generic problem}\label{sec:generic}

In this section that is independent of the previous sections, we develop a generic inclusion problem that can be applied to the main problem (\ref{ib-P}) as a special case. Since the core analysis part is essentially the same for both the Dirichlet problem and Neumann problem, we take the generic approach instead of studying the convex integration in a special setup to avoid repetition when we deal with the Neumann problem in a subsequent paper.

\subsection{Two-wall inclusions}

As a setup, we fix some generic notations and introduce a two-wall partial differential inclusion of inhomogeneous type.

\subsubsection{Related sets}

Let $r_1<r_2$, and let $\omega_1,\omega_2\in C([r_1,r_2])$ be any two functions such that
\[
\max_{[r_1,r_2]}\omega_1 < \min_{[r_1,r_2]}\omega_2.
\]
For each $b\in\R$, define the matrix sets
\[
\begin{split}
K^+(b)=K^+_{\omega_2}(b) & =\bigg\{\begin{pmatrix}
\omega_2(r) & c \\
b & r
\end{pmatrix}\in\M^{2\times 2}\, \Big|\, r\in[r_1,r_2],\, c\in\R \bigg\},\\
K^-(b)=K^-_{\omega_1}(b) & =\bigg\{\begin{pmatrix}
\omega_1(r) & c \\
b & r
\end{pmatrix}\in\M^{2\times 2}\, \Big|\, r\in[r_1,r_2],\, c\in\R \bigg\},\\
U(b)=U_{\omega_1,\omega_2}(b) & =\bigg\{\begin{pmatrix}
s & c \\
b & r
\end{pmatrix}\in\M^{2\times 2}\, \Big|\, r\in(r_1,r_2),\, \omega_1(r)<s<\omega_2(r), \, c\in\R \bigg\},\\
\end{split}
\]
and $K(b)=K_{\omega_1,\omega_2}(b)=K^+(b)\cup K^-(b)$. Let
\[
\begin{split}
K^+=K^+_{\omega_2} & =\{(\omega_2(r),r)\,|\, r\in[r_1,r_2]\}, \\
K^-=K^-_{\omega_1} & =\{(\omega_1(r),r)\,|\, r\in[r_1,r_2]\},
\end{split}
\]
and $K=K_{\omega_1,\omega_2}=K^+\cup K^-$. Also, let
\[
U =\{(s,r)\, |\, r\in(r_1,r_2),\, \omega_1(r)<s<\omega_2(r) \}.
\]

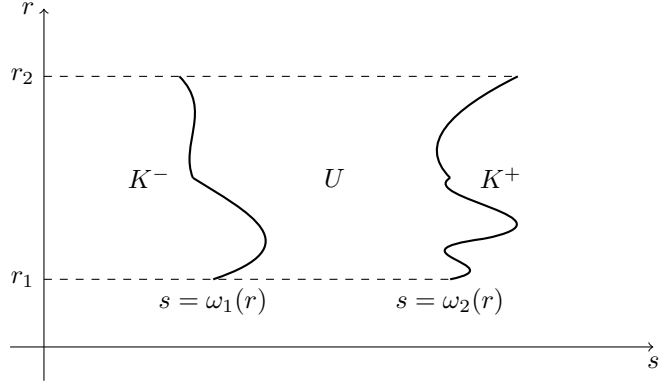
\begin{figure}[ht]
\begin{center}
\begin{tikzpicture}[scale =0.9]
    \draw[->] (-0.5,0) -- (9,0);
    \draw[->] (0,-0.5) -- (0,5);
    \draw[dashed] (0,4)  --  (7,4) ;
    \draw[dashed] (0,1)  --  (6,1) ;
    \draw[thick]   (2, 4) .. controls (2.5,3.5) and  (2,3)   ..(2.2,2.5);
    \draw[thick]   (2.2,2.5) .. controls (3,2) and  (4,1.5)   ..(2.5,1);
    \draw[thick]   (7, 4) .. controls  (6, 3.5) and (5.5,3) ..(6, 2.5 );
    \draw[thick]   (6, 2.5) .. controls  (5.5,2.2) and (8,1.9) ..(6.5,1.6);
    \draw[thick]   (6.5,1.6) .. controls  (5,1.4) and (7,1.2) ..(6,1);
    \draw (0,4) node[left] {$r_2$};
    \draw (0,1) node[left] {$r_1$};
    \draw (0,5) node[left] {$r$};
    \draw (9,0) node[below] {$s$};
    \draw (2.5,1) node[below] {$s=\omega_1(r)$};
    \draw (6,1) node[below] {$s=\omega_2(r)$};
    \draw (2,2.5) node[left] {$K^-$};
    \draw (6.3,2.5) node[right] {$K^+$};
    \draw (4,2.5) node[right] {$U$};
    \end{tikzpicture}
\end{center}
\caption{The right wall $K^+$ and left wall $K^-$.}
\label{fig2-1}
\end{figure}

\subsubsection{Two-wall inclusions}

Let
\[
\Omega_{t_1}^{t_2}=\Omega\times(t_1,t_2)=(0,L)\times(t_1,t_2)\subset\R^2,
\]
where $t_1<t_2$ are any two fixed real numbers, and let $Q\subset \Omega_{t_1}^{t_2}$ be a nonempty open set.
Consider the inhomogeneous partial differential inclusion,
\begin{equation}\label{inclusion}
\nabla z\in K(v)\;\;\mbox{in $Q$},
\end{equation}
where $z=(v,w):Q\to \R^2$.
Regarding this, we fix some terminologies.

\begin{defn}
Let $z=(v,w)\in W^{1,\infty}(Q;\R^2).$ Then the function $z$ is called a \emph{solution} of inclusion (\ref{inclusion}) if
\[
\nabla z\in K(v)\;\;\mbox{a.e. in $Q$},
\]
a \emph{subsolution} of (\ref{inclusion}) if
\[
\nabla z\in K(v)\cup U(v)\;\;\mbox{a.e. in $Q$},
\]
and a \emph{strict subsolution} of (\ref{inclusion}) if
\[
\nabla z\in U(v)\;\;\mbox{a.e. in $Q$},
\]
respectively.
\end{defn}

Observe that if $z=(v,w)\in W^{1,\infty}(Q;\R^2)$ is a solution of (\ref{inclusion}), then
\[
(v_x,w_t)\in K=K^+\cup K^- \;\;\mbox{a.e. in $Q$;}
\]
that is, $(v_x,w_t)$ lies either in the ``right wall'' $K^+$ or in the ``left wall'' $K^-$ almost everywhere in $Q$(see Figure \ref{fig2-1}).

\subsection{Special solutions to generic problem}
Continuing the previous setup, we present an important existence result on inclusion (\ref{inclusion}) that will serve as the main ingredient for proving Theorem \ref{thm:main}.

Assume that $z^\star=(v^\star,w^\star)\in C^1(\bar{\Omega}_{t_1}^{t_2};\R^2)$ is a function such that in $Q,$
\begin{equation}\label{subsolution-1}
w^\star_x = v^\star,\;\; r_1<w^\star_t<r_2, \;\;\mbox{and}\;\; \omega_1(w^\star_t)<v^\star_x<\omega_2(w^\star_t).
\end{equation}
From the definition of $U(b)$ $(b\in\R)$,
\[
\nabla z^\star= \begin{pmatrix}
v^\star_x & v^\star_t \\
w^\star_x & w^\star_t
\end{pmatrix} \in U(v^\star)\;\;\mbox{in $Q$;}
\]
that is, $z^\star$ is a strict subsolution of inclusion (\ref{inclusion}).
In particular, we have
\begin{equation}\label{subsolution-3}
(v^\star_x,w^\star_t)\in U \;\;\mbox{in $Q$.}
\end{equation}
Assume further that
\begin{equation}\label{subsolution-2}
(v^\star_x,w^\star_t)(Q) \cap \{(s,r)\in U\,|\, \dist((s,t),\partial U)<\delta\}\ne\emptyset
\end{equation}
for all sufficiently small $\delta>0.$

We are now ready to state the the main result of this section whose proof is given in section \ref{sec:proof-main-thm}.

\begin{thm}\label{thm:two-wall}
Let $\epsilon>0.$ Then there exists a function $z=(v,w)\in W^{1,\infty}(\Omega_{t_1}^{t_2};\R^2)$ satisfying the following:
\begin{itemize}
\item[(i)] $z$ is a solution of inclusion (\ref{inclusion}),
\item[(ii)] $z=z^\star$\;\;on $\bar{\Omega}_{t_1}^{t_2}\setminus Q,$
\item[(iii)]$\nabla z=\nabla z^\star$\;\;a.e. on $\Omega_{t_1}^{t_2}\cap\partial Q,$
\item[(iv)] $\|z-z^\star\|_{L^\infty(\Omega_{t_1}^{t_2};\R^2)}<\epsilon,$
\item[(v)] $\|v_t-v^\star_t\|_{L^\infty(\Omega_{t_1}^{t_2})}<\epsilon,$
\item[(vi)] for any nonempty open set $O\subset Q,$
\[
\underset{O}{\mathrm{ess\,osc}}\, v_x \ge d_0,
\]
where $d_0:=\min_{[r_1,r_2]}\omega_2-\max_{[r_1,r_2]}\omega_1>0.$
\end{itemize}
\end{thm}

\section{Application of Theorem \ref{thm:two-wall}: proof of main result}\label{sec:setup}

In this section, we turn back to section \ref{sec:mainresult} and prove the main result of the paper, Theorem \ref{thm:main}, by applying Theorem \ref{thm:two-wall}. For the reader's convenience, \underline{we underline} \underline{the arguments in Theorem \ref{thm:main} that are proved along the way.}

To start the proof, assume
\[
M_0> s^-_1,
\]
and fix any two numbers $r_1<r_2$ in $[\rho(s^+_0),r^*]$ such that
\begin{equation}\label{choice-r_1}
s^-(r_1)<M_0.
\end{equation}

In order to fit into the setup in section \ref{sec:generic}, for $r_1\le r\le r_2$, define
\[
\omega_1(r)=s^-(r)\;\;\mbox{and}\;\;\omega_2(r)=s^+(r);
\]
then $\omega_1,\omega_2\in C([r_1,r_2])$, and
\[
\max_{[r_1,r_2]}\omega_1=s^-(r_2)<s^+(r_1)=\min_{[r_1,r_2]}\omega_2.
\]

Next, using elementary calculus, we can choose a function $\rho^\star\in C^3(\R)$ such that
\begin{equation}\label{rho-modi}
\left\{
\begin{array}{l}
  \rho^\star=\rho\;\;\mbox{on $(-\infty,s^-(r_1)]\cup[s^+(r_2),\infty)$,} \\[1mm]
  (\rho^\star)'>0\;\;\mbox{on $[s^-(r_1),s^+(r_2)],$} \\[1mm]
  \rho^\star<\rho\;\;\mbox{on $(s^-(r_1),s^-(r_2)],$}\;\;\mbox{and}\;\;\rho^\star>\rho\;\;\mbox{on $[s^+(r_1),s^+(r_2)).$}
\end{array}
\right.
\end{equation}
Define $\sigma^\star=(\rho^\star)'\in C^2(\R);$ then $\exists c_0>0$ such that $\sigma^\star\ge c_0$ in $\R$. Now, from \cite[Theorem 12.14]{Ln} and \cite[Theorem 1.1]{Ki}, the modified problem,
\begin{equation}\label{ib-modi}
\begin{cases}
u^\star_{t} =(\sigma^\star(u^\star)u^\star_x)_x=(\rho^\star(u^\star))_{xx} & \mbox{in $\Omega_\infty$,}\\
u^\star =u_0 & \mbox{on $\Omega\times \{t=0\}$},\\
u^\star=0 & \mbox{on $\partial\Omega\times(0,\infty)$,}
\end{cases}
\end{equation}
possesses a unique global solution \underline{$u^\star\in C^{2,1}(\bar\Omega_\infty;[0,1])$ with $u^\star\in C^{2+a,1+\frac{a}{2}}(\bar\Omega_T;[0,1])$} \underline{for each $T>0$} such that
\begin{equation}\label{sol-star-prop-1}
0 \le u^\star(x,t)\le\|u^\star(\cdot,s)\|_{L^\infty(\Omega)}\;\;\mbox{for all $x\in\bar\Omega$ and $t\ge s\ge0$}
\end{equation}
and that
\begin{equation}\label{sol-star-prop-2}
\|u^\star(\cdot,t)\|_{L^\infty(\Omega)}\le C e^{-\gamma t}\;\;\mbox{for all $t\ge0$},
\end{equation}
for some constants $C> 0$ and $\gamma>0$, depending only on $M_0$, $L$, $\sigma^\star$, and $(\sigma^\star)'$.
Thus, \underline{(e) and (f) in Theorem \ref{thm:main} are satisfied.} Note also from (\ref{rho-modi}) and (\ref{ib-modi}) that for each $t\ge0$,
\[
\begin{split}
\frac{d}{dt}\int_\Omega u^\star(x,t)\,dx & = \int_\Omega u_t^\star(x,t)\,dx= \sigma^\star(u^\star(L,t))u^\star_x(L,t)-\sigma^\star(u^\star(0,t))u^\star_x(0,t)\\
& = u^\star_x(L,t)-u^\star_x(0,t)\le0
\end{split}
\]
as $u^\star(x,t)\ge0$ for $0<x<L.$ So \underline{the second of (d)  in Theorem \ref{thm:main} holds.}

We define
\begin{equation}\label{def:vstar}
v^\star(x,t)=\int_0^x u^\star(y,t)\,dy+ \int_0^t u^\star_x(0,s)\,ds\;\;\forall (x,t)\in\bar\Omega_\infty;
\end{equation}
then from (\ref{ib-modi}) and the choice of $\rho^\star$, $v^\star\in C^{3,1}(\bar\Omega_\infty)$ satisfies that for all $(x,t)\in\Omega_\infty$,
\[
\begin{split}
v^\star_t(x,t) & = \int_0^x u^\star_t(y,t)\,dy+ u^\star_x(0,t) = \int_0^x (\sigma^\star(u^\star)u^\star_x)_x(y,t)\,dy+ u^\star_x(0,t)\\
& = \sigma^\star(u^\star(x,t))u^\star_x(x,t) - \sigma^\star(u^\star(0,t))u^\star_x(0,t) + u^\star_x(0,t)\\
& = \sigma^\star(v^\star_x(x,t))v^\star_{xx}(x,t) = (\rho^\star(v^\star_x))_{x}(x,t).
\end{split}
\]
Hence, $v^\star$ is a global solution to the problem,
\begin{equation}\label{ib-modi-1}
\begin{cases}
v^\star_{t} =(\rho^\star(v^\star_x))_{x} & \mbox{in $\Omega_\infty$,}\\
v^\star =v_0 & \mbox{on $\Omega\times \{t=0\}$},\\
v^\star_x=0 & \mbox{on $\partial\Omega\times(0,\infty)$,}
\end{cases}
\end{equation}
where
\[
v_0(x):=\int_0^x u_0(y)\,dy \;\;\forall x\in\bar\Omega.
\]

In turn, we define
\[
w^\star(x,t)=\int_0^t \rho^\star (v^\star_x(x,s))\,ds + \int_0^x v_0(y)\,dy\;\;\forall (x,t)\in\bar\Omega_\infty;
\]
then from (\ref{def:vstar}) and (\ref{ib-modi-1}),
\begin{equation}\label{system-1}
\begin{cases}
w^\star_{x} =v^\star \\
w^\star_t =\rho^\star(v^\star_x)
\end{cases}
\;\;\mbox{in $\Omega_\infty$},
\end{equation}
$v^\star_x=u^\star\in C^{2,1}(\bar\Omega_\infty;[0,1])$, and $z^\star:=(v^\star,w^\star)\in (C^{3,1}\times C^{4,2})(\bar\Omega_\infty).$

Define
\[
Q=\{(x,t)\in\Omega_\infty\,|\, s^-(r_1)<v^\star_x(x,t)<s^+(r_2)\}.
\]
Fix a number $T_0>0$ so large that $Ce^{-\gamma T_0}< s^-(r_1)$. Then from (\ref{sol-star-prop-1}) and (\ref{sol-star-prop-2}),
\[
v^\star_x(x,t)=u^\star(x,t)< s^-(r_1)\;\;\forall (x,t)\in\Omega\times[T_0,\infty);
\]
thus,
\[
Q\subset \Omega^{T_0}_0=\Omega\times(0,T_0).
\]
Since $u_0(0)=u_0(L)=0$ and $M_0>s^-(r_1)$, we can take a point $x_0\in\Omega$ such that
\[
s^-(r_1)<v^\star_x(x_0,0)=u_0(x_0)<s^+(r_2).
\]
Then by continuity, we can take an $r_0\in(0,T_0)$ with $r_0<\min\{x_0,L-x_0\}$ so small that
\[
s^-(r_1)<v^\star_x(x,t)<s^+(r_2)\;\;\forall(x,t)\in (\Omega\times[0,\infty)) \cap B_{r_0}(x_0,0);
\]
thus, $\Omega_\infty \cap B_{r_0}(x_0,0)\subset Q\ne\emptyset$ so that
\[
(x_0-r_0,x_0+r_0)\times\{0\}\subset\underline{\bar{Q}\cap(\Omega\times\{0\})\ne\emptyset.}
\]
Note here that \underline{$Q$ is a nonempty bounded open subset of $\Omega_\infty.$} Observe also that
\[
\underline{\bar{Q}\subset\Omega\times[0,\infty)}
\]
as $\bar{Q}$ is compact and $v^\star_x(0,t)=v^\star_x(L,t)=0$ for all $t\ge0$.

Following the notations in section \ref{sec:generic}, note from the choice of $\rho^\star$, the definition of $Q$, (\ref{choice-r_1}), (\ref{ib-modi-1}), and (\ref{system-1}) that (\ref{subsolution-1}) holds in $Q$ and (\ref{subsolution-2}) holds for all sufficiently small $\delta>0$.
Thus, for any fixed $\epsilon>0,$ we can apply Theorem \ref{thm:two-wall} (with $t_1=0$ and $t_2=T_0$ in the current case) to obtain a function $z=z_\epsilon\in W^{1,\infty}(\Omega^{T_0}_0;\R^2)$ satisfying the following, where $z=(v,w)$:
\begin{itemize}
\item[(i)] $z$ is a solution of inclusion (\ref{inclusion});
\item[(ii)] $z=z^\star$\;\;on $\bar{\Omega}^{T_0}_0\setminus Q;$
\item[(iii)]$\nabla z=\nabla z^\star$\;\;a.e. on $\Omega^{T_0}_0\cap\partial Q;$
\item[(iv)] $\|z-z^\star\|_{L^\infty(\Omega^{T_0}_0;\R^2)}<\epsilon;$
\item[(v)] $\|v_t-v^\star_t\|_{L^\infty(\Omega^{T_0}_0)}<\epsilon;$
\item[(vi)] for any nonempty open set $O\subset Q,$
\[
\underset{O}{\mathrm{ess\,osc}}\, v_x \ge d_0,
\]
where $d_0=\min_{[r_1,r_2]}\omega_2-\max_{[r_1,r_2]}\omega_1=s^+(r_1)-s^-(r_2)>0.$
\end{itemize}
For each $(x,t)\in\bar{\Omega}\times [T_0,\infty)$, define
\begin{equation}\label{def:z-1}
z(x,t)=z^\star(x,t);
\end{equation}
then from (ii), $z=(v,w)\in W^{1,\infty}(\Omega_T;\R^2)$ for all $T>0$.

For each $t\ge 0,$ define
\begin{equation}\label{def:u-sol}
u(\cdot,t)=v_x(\cdot,t)\;\;\mbox{a.e. in $\Omega$};
\end{equation}
then $u\in L^\infty(\Omega_T)$ for all $T>0$.
First, from (vi), \underline{(c) in Theorem \ref{thm:main} is fulfilled.}
Also, note from (ii) and (\ref{def:vstar}) that for every $t\ge 0,$
\[
\begin{split}
\int_\Omega u(x,t)\,dx & =\int_\Omega v_x(x,t)\,dx=v(L,t)-v(0,t) \\
&= v^\star(L,t)-v^\star(0,t) =\int_\Omega v^\star_x(x,t)\,dx= \int_\Omega u^\star(x,t)\,dx;
\end{split}
\]
hence, \underline{the first of (d) in Theorem \ref{thm:main} holds.}
From (ii), (\ref{def:vstar}), and (\ref{def:z-1}), we have
\[
u=v_x=v^\star_x=u^\star\in[0,1]\;\;\mbox{in $\Omega_\infty\setminus\bar{Q}$},
\]
and from (iii) and (\ref{def:vstar}), we get
\[
u=v_x=v^\star_x=u^\star\in[0,1]\;\;\mbox{a.e. in $\Omega_\infty\cap\partial Q$}.
\]
In particular, \underline{(a) in Theorem \ref{thm:main} is satisfied.}
From (i), we have
\[
\nabla z\in K(v)\;\;\mbox{a.e. in $Q$;}
\]
that is, a.e. in $Q$,
\begin{equation}\label{property-v-w}
\left\{
\begin{array}{l}
  u=v_x\in[s^-(r_1),s^-(r_2)]\cup[s^+(r_1),s^+(r_2)],  \\
  w_t=\rho(v_x), \\
  w_x=v.
\end{array}
\right.
\end{equation}
In particular, we see that $u\in[0,1]$ a.e. in $\Omega_\infty$, that is, \underline{$u\in L^\infty(\Omega_\infty;[0,1])$} and that \underline{(b) in  Theorem \ref{thm:main} holds.}

Now, we check that $u$ is a global weak solution to problem (\ref{ib-P}). To do so, fix any $T>0$ and any test function $\varphi\in C^\infty(\bar\Omega\times[0,T])$ with
\begin{equation}\label{bdry:testfn}
\varphi=0\;\;\mbox{on $(\partial\Omega\times[0,T])\cup(\Omega\times\{t=T\}).$}
\end{equation}
Observe from (ii), (\ref{ib-modi-1}), (\ref{def:z-1}), (\ref{def:u-sol}), (\ref{property-v-w}), (\ref{bdry:testfn}), the choice of $\rho^\star$, and the integration by parts that
\[
\begin{split}
\int_0^T & \int_0^L(u\varphi_t+\rho(u)\varphi_{xx})\,dxdt = \int_0^T\int_0^L (u\varphi_t+w_t\varphi_{xx})\,dxdt \\
= &  \int_0^T w^\star_t(L,t)\varphi_{x}(L,t)\,dt -\int_0^T w^\star_t(0,t)\varphi_{x}(0,t)\,dt + \int_0^L w^\star_x(x,0)\varphi_{x}(x,0)\,dx\\
= &  \int_0^T \rho^\star(v^\star_x(L,t))\varphi_{x}(L,t)\,dt -\int_0^T \rho^\star(v^\star_x(0,t))\varphi_{x}(0,t)\,dt + \int_0^L v_0(x)\varphi_{x}(x,0)\,dx \\
= & -\int_0^L v_0'(x)\varphi(x,0)\,dx +v_0(L)\varphi(L,0)-v_0(0)\varphi(0,0)=-\int_0^L u_0(x)\varphi(x,0)\,dx.
\end{split}
\]
Thus, according to Definition \ref{def:global-weak-sol}(i), \underline{$u$ is a global weak solution to (\ref{ib-P}).}

Since $u^\star$ itself is not a global weak solution to problem (\ref{ib-P}), it follows from (iv) that \underline{there are infinitely many global weak solutions to (\ref{ib-P}) that satisfy properties} \underline{(a)--(f) in Theorem \ref{thm:main}.}

The proof of Theorem \ref{thm:main} is now complete.

\section{Proof of Theorem \ref{thm:two-wall}}\label{sec:proof-main-thm}

Following section \ref{sec:generic}, this section presents a proof of Theorem \ref{thm:two-wall} with the help of the key lemma, Lemma \ref{lem:main}, to be proved in section \ref{sec:proof-main-lem}.

\subsection{Selection of an in-approximation to $K$}\label{subsection:selection}

For each $b\in\R$ and $0\le\lambda<\frac{1}{2}$, define the matrix set $U^\lambda(b)=U^\lambda_{\omega_1,\omega_2}(b)\subset\M^{2\times 2}$ by
\[
U^\lambda(b)=\bigg\{\begin{pmatrix}
s & c \\
b & r
\end{pmatrix}\, \Big| \begin{array}{l}
                                         c\in\R,\,(1-\lambda)r_1+\lambda r_2< r <\lambda r_1+(1-\lambda) r_2, \\[0.5mm]
                                         (1-\lambda)\omega_1(r)+\lambda \omega_2(r)<s<\lambda\omega_1(r)+(1-\lambda)\omega_2(r)
                                       \end{array}\bigg\},
\]
and let $U^\lambda=U^\lambda_{\omega_1,\omega_2}\subset\R^2$ be given by
\[
U^\lambda=\bigg\{(s,r)\, \Big| \begin{array}{l}
                                         (1-\lambda)r_1+\lambda r_2< r <\lambda r_1+(1-\lambda) r_2, \\[0.5mm]
                                         (1-\lambda)\omega_1(r)+\lambda \omega_2(r)<s<\lambda\omega_1(r)+(1-\lambda)\omega_2(r)
                                       \end{array}\bigg\}.
\]
Observe that $U^0(b)=U(b)$ for every $b\in\R$ and that $U^0=U.$

Since $|Q|<\infty$, we can select a sequence $\{\lambda_i\}_{i\in\N}$ in $\R$ with
\begin{equation}\label{choice-lamda-i}
\frac{1}{(2i+1)^2}<\lambda_i<\frac{1}{(2i)^2}<\frac{1}{2}\;\;\forall i\in\N
\end{equation}
such that for every $i\in\N,$
\[
\big|\{(x,t)\in Q\,|\, (v^\star_x(x,t),w^\star_t(x,t))\in \partial U^{\lambda_i}\}\big|=0.
\]
Thanks to (\ref{subsolution-2}), we may assume
\[
Q_i:=\{(x,t)\in Q\,|\, (v^\star_x(x,t),w^\star_t(x,t))\in U^{\lambda_i}\setminus\bar{U}^{\lambda_{i-1}}\}\ne\emptyset\;\;\forall i\in\N,
\]
where $U^{\lambda_0}:=\emptyset;$ then from (\ref{subsolution-3}), $\{Q_i\}_{i\in\N}$ is a sequence of disjoint  open subsets of $Q$ whose union has measure $|Q|.$

For each $i\in\N,$ let $\lambda_i'=\frac{\lambda_i+\lambda_{i+1}}{2}$. Observe from (\ref{choice-lamda-i}) that for all $i\in\N,$
\[
\begin{split}
\frac{\lambda_i-\lambda_{i+1}'}{\lambda_{i+1}-\lambda_{i+1}'} & = \frac{\lambda_i-\frac{\lambda_{i+1}+\lambda_{i+2}}{2}}{\lambda_{i+1}-\frac{\lambda_{i+1}+\lambda_{i+2}}{2}} =\frac{2\lambda_i-\lambda_{i+1}-\lambda_{i+2}}{\lambda_{i+1}-\lambda_{i+2}} \\
& \le \frac{\frac{2}{(2i)^2}-\frac{1}{(2(i+1)+1)^2}-\frac{1}{(2(i+2)+1)^2}}{\frac{1}{(2(i+1)+1)^2}-\frac{1}{(2(i+2))^2}}\\
& =\frac{2(64i^3+308i^2+480i+225)(i+2)^2}{i^2(2i+5)^2(4i+7)}\le 36
\end{split}
\]
and that
\[
\begin{split}
\frac{\lambda_{i}-\lambda_{i+1}'}{1-2\lambda_{i+1}'} & =\frac{\lambda_{i}-\frac{\lambda_{i+1}+\lambda_{i+2}}{2}}{1-\lambda_{i+1}-\lambda_{i+2}} =\frac{1}{2}\cdot\frac{2\lambda_i-\lambda_{i+1}-\lambda_{i+2}}{1-\lambda_{i+1}-\lambda_{i+2}}\\
& \le \frac{1}{2}\cdot\frac{\frac{2}{(2i)^2}-\frac{1}{(2(i+1)+1)^2}-\frac{1}{(2(i+2)+1)^2}}{1-\frac{1}{(2(i+1))^2}-\frac{1}{(2(i+2))^2}} \\
& =\frac{(i+1)^2(i+2)^2(64i^3+308i^2+480i+225)}{i^2(2i+3)^2(2i+5)^2(4i^4+24i^3+50i^2+42i+11)}.
\end{split}
\]
From these observations, we deduce that for each $i\in\N,$
\begin{equation}\label{choice-alpha-i}
0<\alpha_i:=\frac{\lambda_{i}-\lambda_{i+1}'}{1-2\lambda_{i+1}'} \le 36 \frac{\lambda_{i+1}-\lambda_{i+1}'}{1-2\lambda_{i+1}'}
\end{equation}
and that
\[
0<\ell_0:=\sum_{i=1}^\infty \alpha_i<1.
\]

Fix a number $\kappa_0\in(1,1/\ell_0)$, and define
\[
\beta_1=1-\frac{1}{\kappa_0};
\]
then
\[
\ell_0<\kappa_0\ell_0<1\;\;\mbox{and}\;\;0< \beta_1=1-1/\kappa_0<1-\ell_0<1.
\]
In turn, for each $i\in\N$, define
\begin{equation}\label{choice-beta_i}
\beta_{i+1}=\beta_i+\frac{\alpha_i}{\kappa_0\ell_0};
\end{equation}
then the sequence $\{\beta_i\}_{i\in\N}$ fulfills that
\[
0< \beta_1<\beta_2<\cdots<1,\;\;\lim_{i\to\infty}\beta_i=1,
\]
and
\[
\frac{\lambda_{i}-\lambda_{i+1}'}{1-2\lambda_{i+1}'}=\kappa_0\ell_0(\beta_{i+1}-\beta_i) \le 36 \frac{\lambda_{i+1}-\lambda_{i+1}'}{1-2\lambda_{i+1}'}\;\;\forall i\in\N.
\]

For each $i\in\N,$ let
\[
\begin{split}
U^+_i & =\bigg\{(s,r)\, \Big| \begin{array}{l}
                                         (1-\lambda_{i+1})r_1+\lambda_{i+1} r_2< r <\lambda_{i+1} r_1+(1-\lambda_{i+1}) r_2, \\[0.5mm]
                                         \lambda_{i}\omega_1(r)+(1-\lambda_{i})\omega_2(r)<s<\lambda_{i+1}\omega_1(r)+(1-\lambda_{i+1})\omega_2(r)
                                       \end{array}\bigg\}, \\
U^-_i & =\bigg\{(s,r)\, \Big| \begin{array}{l}
                                         (1-\lambda_{i+1})r_1+\lambda_{i+1} r_2< r <\lambda_{i+1} r_1+(1-\lambda_{i+1}) r_2, \\[0.5mm]
                                         (1-\lambda_{i+1})\omega_1(r)+\lambda_{i+1}\omega_2(r)<s<(1-\lambda_{i})\omega_1(r)+\lambda_{i}\omega_2(r)
                                       \end{array}\bigg\}.
\end{split}
\]
Next, for each $i\in\N,$ let
\[
I_{i+1}=[(1-\lambda_{i+1})r_1+\lambda_{i+1}r_2,\lambda_{i+1}r_1+(1-\lambda_{i+1})r_2]\subset\R,
\]
\[
\begin{split}
\eta^{1,+}_i & =\min_{r,\bar{r}\in I_{i+1}}|(\lambda_i'\omega_1(r)+(1-\lambda_i')\omega_2(r),r)-(\lambda_i\omega_1(\bar{r})+(1-\lambda_i)\omega_2(\bar{r}),\bar{r})|,\\
\eta^{2,+}_i & =\min_{r,\bar{r}\in I_{i+1}}|(\lambda_i'\omega_1(r)+(1-\lambda_i')\omega_2(r),r)-(\lambda_{i+1}\omega_1(\bar{r})+(1-\lambda_{i+1})\omega_2(\bar{r}),\bar{r})|,\\
\eta^{1,-}_i & =\min_{r,\bar{r}\in I_{i+1}}|((1-\lambda_i')\omega_1(r)+\lambda_i'\omega_2(r),r)-((1-\lambda_i)\omega_1(\bar{r})+\lambda_i\omega_2(\bar{r}),\bar{r})|,\\
\eta^{2,-}_i & =\min_{r,\bar{r}\in I_{i+1}}|((1-\lambda_i')\omega_1(r)+\lambda_i'\omega_2(r),r)-((1-\lambda_{i+1})\omega_1(\bar{r})+\lambda_{i+1}\omega_2(\bar{r}),\bar{r})|,\\
\end{split}
\]
and
\[
\eta_i=\min\{\eta^{1,+}_i,\eta^{2,+}_i,\eta^{1,-}_i,\eta^{2,-}_i,(\lambda_i-\lambda_{i+1})(r_2-r_1) \}>0.
\]

\subsection{Main lemma}

Define
\[
\zeta^\pm_0(s,r)=\theta_0(\mathrm{dist}((s,r),K^\pm))\;\;\;((s,r)\in\R^2),
\]
where $\theta_0\in C^\infty([0,\infty))$ is a cutoff function such that $0\le\theta_0\le 1$ on $[0,\infty),$ $\theta_0=1$ on $[0,d_0/4],$ and $\theta_0=0$ on $[d_0/2,\infty).$ Then we can choose an integer $i_0\ge 2$ so large that for all $i\ge i_0,$
\[
\zeta_0^\pm (s,r)=1\;\;\;\forall(s,r)\in U^\pm_{i-1},
\]
respectively.

We present here the key lemma to finish the proof of Theorem \ref{thm:two-wall}. A proof of this lemma is provided in section \ref{sec:proof-main-lem}.

\begin{lem}\label{lem:main}
Let $\epsilon>0.$ Then there exist a sequence $\{z_i\}_{i\in\N}=\{(v_i,w_i)\}_{i\in\N}$ in $W^{1,\infty}(\Omega_{t_1}^{t_2};\R^2)$, three sequences $\{\mathcal{F}_i\}_{i\in\N}$ and $\{\mathcal{F}^\pm_i\}_{i\in\N}$ as follows, and three positive constants $c_1$, $c_2$, and $c_3$ such that for each $i\in\N,$ the following are satisfied:
\begin{itemize}
\item[(a)] $\mathcal{F}_i=\{D_{ij}\,|\, j\in\N\}$ is a countable collection of disjoint rhombic domains in $Q_1\cup\cdots\cup Q_i;$
\item[(b)] $\mathcal{F}^+_i=\{T_{ij}\,|\, j\in\N\}$ is a countable collection of disjoint triangular domains in $Q_1\cup\cdots\cup Q_i;$
\item[(c)] $\mathcal{F}^-_i=\{R_{ij}\,|\, j\in\N\}$ is a countable collection of disjoint rhombic domains in $Q_1\cup\cdots\cup Q_i;$
\item[(d)]
$
(\cup_{j\in\N}T_{ij})\cap(\cup_{j\in\N}R_{ij})=\emptyset
$
and
\[
|Q_1\cup\cdots\cup Q_i|=|\cup_{j\in\N}D_{ij}|=|\cup_{j\in\N}(T_{ij}\cup R_{ij})|;
\]
\item[(e)] for each $D\in\mathcal{F}_i$, there are four disjoint  triangular domains $T^1_D,T^2_D,T^3_D,T^4_D\in\mathcal{F}^+_i$ and one rhombic domain $R_D\in\mathcal{F}^-_i$ such that
\[
T^1_D\cup T^2_D\cup T^3_D\cup T^4_D\cup R_D\subset D\;\;\mbox{and}\;\; |D|=|T^1_D\cup T^2_D\cup T^3_D\cup T^4_D\cup R_D|;
\]
\item[(f)] one has
\[
\sup_{j\in\N}\mathrm{diam}\,D_{ij}\le\frac{1}{2^{i}};
\]
\item[(g)] one has
\[
\left\{
\begin{array}{l}
  \mbox{for each $T\in \mathcal{F}^+_i,$ $z_i\in C^1(\bar{T};\R^2),$ and $((v_i)_x,(w_i)_t)\in U^+_i$ in $T,$} \\
  \mbox{for each $R\in \mathcal{F}^-_i,$ $z_i\in C^1(\bar{R};\R^2),$ and $((v_i)_x,(w_i)_t)\in U^-_i$ in $R,$} \\
  \mbox{$z_i=z_0$ on $\bar{\Omega}_{t_1}^{t_2}\setminus (Q_1\cup \cdots\cup Q_i),$}  \\
  \mbox{$\nabla z_i=\nabla z_0$ a.e. on $\Omega_{t_1}^{t_2}\cap\partial Q,$}
\end{array}
\right.
\]
where $z_0=(v_0,w_0):=z^\star;$
\item[(h)] one has
\[
\left\{
\begin{array}{l}
  \mbox{$(w_i)_x=v_i$ a.e. in $Q;$} \\
  \|(v_i)_t-(v_{i-1})_t\|_{L^\infty(\Omega_{t_1}^{t_2})}\le\frac{\epsilon}{2^{i+1}}; \\
  \|z_i-z_{i-1}\|_{L^\infty(\Omega_{t_1}^{t_2};\R^2)}\le\frac{\epsilon}{2^{i+1}}; \\
  |z_i(x,t)-z_i(y,s)|\le c_1|(x,t)-(y,s)|\;\;\forall(x,t),(y,s)\in \Omega_{t_1}^{t_2};
\end{array}
\right.
\]
\item[(i)] if $i\ge 2,$ then
\[
\int_{\Omega_{t_1}^{t_2}}|\nabla z_i-\nabla z_{i-1}|\,dxdt\le c_2((\beta_i-\beta_{i-1})|Q|+|Q_i|);
\]
\item[(j)] if $i\ge i_0$ and $D\in\mathcal{F}_{i-1}$, then
\[
\begin{split}
\int_D \zeta^\pm_0((v_i)_x,(w_i)_t)\,dxdt & \ge c_3(\beta_i-\beta_{i-1})|D|\;\;\mbox{and}\\
\int_D \zeta^\pm_0((v_i)_x,(w_i)_t)\,dxdt & \ge (1-(\beta_i-\beta_{i-1}))\int_D \zeta^\pm_0((v_{i-1})_x,(w_{i-1})_t)\,dxdt.
\end{split}
\]
\end{itemize}
\end{lem}

\subsection{Completion of proof}\label{subsection:completion}
Utilizing Lemma \ref{lem:main}, we now complete the proof of Theorem \ref{thm:two-wall}.

From the third of (h) in Lemma \ref{lem:main}, it follows that for all $i> j\ge1$,
\[
\|z_i-z_j\|_{L^\infty(\Omega_{t_1}^{t_2};\R^2)}\le \epsilon\Big(\frac{1}{2^{i+1}}+\cdots+\frac{1}{2^{j+2}}\Big)\le\frac{\epsilon}{2^{j+1}};
\]
thus, $\{z_i\}_{i\in\N}$ is a Cauchy sequence in $L^\infty(\Omega_{t_1}^{t_2};\R^2)$ so that from the fourth of (h) in Lemma \ref{lem:main},
\begin{equation}\label{completion-1}
z_i\to z\;\;\mbox{in $L^\infty(\Omega_{t_1}^{t_2};\R^2)$ as $i\to\infty,$}
\end{equation}
for some $z=(v,w)\in W^{1,\infty}(\Omega_{t_1}^{t_2};\R^2)$ satisfying
\[
|z(x,t)-z(y,s)|\le c_1|(x,t)-(y,s)|\;\;\forall(x,t),(y,s)\in \Omega_{t_1}^{t_2}.
\]

Convergence (\ref{completion-1}) together with the third of (g) and third of (h) in Lemma \ref{lem:main}  implies that
\[
z=z_0=z^\star\;\;\mbox{on $\bar{\Omega}_{t_1}^{t_2}\setminus Q$}
\]
and
\[
\|z-z^\star\|_{L^\infty(\Omega_{t_1}^{t_2};\R^2)}= \|z-z_0\|_{L^\infty(\Omega_{t_1}^{t_2};\R^2)}\le\frac{\epsilon}{2}<\epsilon;
\]
that is, (ii) and (iv) in Theorem \ref{thm:two-wall} hold.

From (i) in Lemma \ref{lem:main}, we have
\[
\|\nabla z_i-\nabla z_j\|_{L^1(\Omega_{t_1}^{t_2};\M^{2\times 2})}\le c_2\bigg(|\beta_i-\beta_j||Q|+\sum_{k=\min\{i,j\}+1}^\infty|Q_k|\bigg)\to 0
\]
as $i,j\to\infty;$ thus, after passing to a subsequence if necessary, for a.e. $(x,t)\in \Omega_{t_1}^{t_2},$
\begin{equation}\label{completion-2}
\nabla z_i(x,t) \to \nabla z(x,t)\;\;\mbox{in $\M^{2\times 2}$ as $i\to\infty.$}
\end{equation}
In particular, this pointwise convergence holds for a.e. $(x,t)\in \Omega_{t_1}^{t_2}\cap\partial Q.$ Thus, (iii) in Theorem \ref{thm:two-wall} follows from the fourth of (g) in Lemma \ref{lem:main}.

From (\ref{completion-1}), (\ref{completion-2}), and the first and second of (h) in Lemma \ref{lem:main}, we have
\begin{equation}\label{completion-3}
w_x=v\;\;\mbox{a.e. in $Q$}
\end{equation}
and
\[
\|v_t-v^\star_t\|_{L^\infty(\Omega_{t_1}^{t_2})}=\|v_t-(v_0)_t\|_{L^\infty(\Omega_{t_1}^{t_2})}\le\frac{\epsilon}{2}<\epsilon;
\]
that is, (v) in Theorem \ref{thm:two-wall} is fulfilled.

For each $i\in\N,$ observe from (b), (c), (d), and the first and second of  (g) in Lemma \ref{lem:main} and from the definition of $U^\pm_i$ and $K=K^+\cup K^-$ that
\[
\begin{split}
\int_{Q}\mathrm{dist}(((v_i)_x & ,(w_i)_t),K)\,dxdt \\
= & \Big(\sum_{\ell=1}^i + \sum_{\ell=i+1}^\infty\Big)\int_{Q_\ell}\mathrm{dist}(((v_i)_x,(w_i)_t),K)\,dxdt \\
\le & S_M|Q|\lambda_i + d_1 \sum_{\ell=i+1}^\infty|Q_\ell|,
\end{split}
\]
where $d_1:=\mathrm{diam} \,U$ and
\[
S_M:=\max_{[r_1,r_2]}(\omega_2-\omega_1) >0.
\]
Here, letting $i\to\infty,$ we obtain from (\ref{completion-2}) that
\[
\int_{Q}\mathrm{dist}((v_x  ,w_t),K)\,dxdt=0;
\]
that is,
\begin{equation}\label{completion-4}
(v_x  ,w_t)\in K\;\;\mbox{a.e. in $Q$}
\end{equation}
as $K\subset \R^2$ is compact. This inclusion together with (\ref{completion-3}) and the definition of $K(b)$ $(b\in\R)$ implies that
\[
\nabla z\in K(v)\;\;\mbox{a.e. in $Q$;}
\]
hence, (i) in Theorem \ref{thm:two-wall} holds.

Finally, to verify (vi) in Theorem \ref{thm:two-wall}, choose an integer $i_1\ge i_0$ so large that for all $i\ge i_1,$
\[
\lim_{n\to\infty}(1-(\beta_n-\beta_{n-1}))(1-(\beta_{n-1}-\beta_{n-2})) \cdots (1-(\beta_{i+1}-\beta_{i}))\ge\frac{1}{2}.
\]
Fix any integer $i\ge i_1$, and let $D\in\mathcal{F}_{i-1}.$
Then note from (\ref{completion-2}) and (j) in Lemma \ref{lem:main} that
\[
\begin{split}
\int_{D} & \zeta_0^\pm(v_x,w_t)\,dxdt =\lim_{n\to\infty} \int_{D}\zeta_0^\pm((v_n)_x,(w_n)_t)\,dxdt \\
& \ge \underset{n\to\infty}{\mathrm{lim\,sup}}\,(1-(\beta_n-\beta_{n-1})) \int_{D}\zeta_0^\pm((v_{n-1})_x,(w_{n-1})_t)\,dxdt \\
& \ge \underset{n\to\infty}{\mathrm{lim\,sup}}\,(1-(\beta_n-\beta_{n-1})) (1-(\beta_{n-1}-\beta_{n-2})) \int_{D}\zeta_0^\pm((v_{n-2})_x,(w_{n-2})_t)\,dxdt
\end{split}
\]
\[
\begin{split}
& \quad\vdots \\
& \ge \underset{n\to\infty}{\mathrm{lim\,sup}}\,(1-(\beta_n-\beta_{n-1})) \cdots (1-(\beta_{i+1}-\beta_{i})) \int_{D}\zeta_0^\pm((v_{i})_x,(w_{i})_t)\,dxdt \\
& \ge \frac{c_3}{2}(\beta_i-\beta_{i-1})|D|>0.
\end{split}
\]

Next, let $O\subset Q$ be any nonempty open set. Since $\{Q_i\}_{i\in\N}$ is a Vitali cover of $Q$, we have $O\cap Q_{i_2}\ne\emptyset$ for some $i_2\in\N.$
Thus, from (a), (d), and (f) in Lemma \ref{lem:main}, there exist an $i_3>\max\{i_1,i_2\}$ and a $D_3\in\mathcal{F}_{i_3 -1}$ such that
\[
D_3\subset O\cap Q_{i_2}.
\]
This inclusion and the above positivity estimate imply that
\[
\int_{O} \zeta_0^\pm(v_x,w_t)\,dxdt\ge \int_{D_3} \zeta_0^\pm(v_x,w_t)\,dxdt>0.
\]
So from the definition of $\zeta_0^\pm$, there are two disjoint sets $O^\pm\subset O$ of positive measure such that
\[
\mathrm{dist}((v_x,w_t),K^\pm)<\frac{d_0}{2}\;\;\mbox{in $O^\pm;$}
\]
thus, with (\ref{completion-4}), we conclude that
\[
(v_x,w_t)\in K^\pm\;\;\mbox{a.e. in $O^\pm.$}
\]
Therefore, (vi) in Theorem \ref{thm:two-wall} follows from the definition of $K^\pm.$

The proof of Theorem \ref{thm:two-wall} is now complete.

\section{Proof of Lemma \ref{lem:main}}\label{sec:proof-main-lem}

This final section is entirely devoted to a proof of Lemma \ref{lem:main}. First, we begin with an elementary lemma in subsection \ref{subsec:patching-function} that provides us with a building block for our constructions. Then we perform inductive surgeries in subsection \ref{subsec:surgeries} that complete the proof of Lemma \ref{lem:main}.

\subsection{Piecewise affine maps with boundary trace $0$}\label{subsec:patching-function}

Although the following theorem is essential in our constructions, its proof is so elementary that we omit it (see \cite{Zh1}).

\begin{lem}\label{lem:patching-function}
Let $\tau^\pm$ and $\delta$ be any three positive numbers, and let $D=D_{\delta}$ be the interior of the convex hull of the four points $(\pm\delta,0)$ and $(0,\pm 1)$ in $\R^2;$ that is,
\[
D=\mathrm{int\,co}\{(\delta,0),\, (-\delta,0),\, (0,1),\,(0,-1)\}.
\]
Let $\varphi=\varphi_{\tau^+,\tau^-,\delta}: \bar{D}\to\R$ be the piecewise affine map defined as follows: for  each $0\le t\le 1,$ define
\[
\varphi(x,t)=\left\{\begin{array}{ll}
                     \tau^+(x+\delta-\delta t), & \mbox{$-\delta+\delta t\le x\le \frac{\tau^+}{\tau^+ +\tau^-}(-\delta+\delta t),$} \\[1.5mm]
                     -\tau^- x, & \mbox{$\frac{\tau^+}{\tau^+ +\tau^-}(-\delta+\delta t) \le x \le \frac{\tau^+}{\tau^+ +\tau^-}(\delta- \delta t),$} \\[1.5mm]
                     \tau^+(x-\delta+ \delta t), & \mbox{$\frac{\tau^+}{\tau^+ +\tau^-}(\delta- \delta t) \le x\le \delta- \delta t;$}
                   \end{array}
 \right.
\]
and for each $(x,t)\in\bar{D}$ with $t\le0,$ define $\varphi(x,t)=\varphi(x,-t).$
Then $\varphi$ satisfies the following:
\begin{itemize}
\item[(i)] $\varphi\in W^{1,\infty}_0(D);$
\item[(ii)] there are five disjoint domains $T^1,T^2,T^3,T^4,R\subset D$ which cover $D$ in the sense of Vitali, and $\varphi$ is affine in each of $T^1,T^2,T^3,T^4,R;$
\item[(iii)] after a proper ordering of $T^1,T^2,T^3,T^4,R,$
\[\nabla\varphi = \left\{\begin{array}{ll}
                                                (\tau^+,\delta\tau^+) & \mbox{in $T^1\cup T^2,$} \\[1.5mm]
                                                (\tau^+,-\delta\tau^+) & \mbox{in $T^3\cup T^4,$} \\[1.5mm]
                                                (-\tau^-,0) & \mbox{in $R ;$}
                                              \end{array}
 \right.\]
\item[(iv)] for each $t\in[-1,1],$
\[
\int_{-\delta+\delta|t|}^{\delta-\delta|t|}\varphi(x,t)\,dx=0;
\]
\item[(v)] $\|\varphi\|_{L^\infty(D)}=\frac{\tau^+\tau^-}{\tau^++\tau^-}\delta;$
\item[(vi)] $\left\{\begin{array}{l}
                     \big|\{(x,t)\in D\,|\,\varphi_x(x,t)=\tau^+ \} \big|=\frac{\tau^-}{\tau^++\tau^-}|D|, \\[1.5mm]
                     \big|\{(x,t)\in D\,|\,\varphi_x(x,t)=-\tau^- \} \big|=\frac{\tau^+}{\tau^++\tau^-}|D|.
                   \end{array}
 \right.$
\end{itemize}
\end{lem}

\begin{figure}[ht]
\begin{center}
\begin{tikzpicture}[scale = 1]
    \draw[->] (6.5,2.5) -- (10,2.5);
    \draw[-] (0,2.5) -- (3.5,2.5);
	\draw[->] (5,4.5) -- (5,5);
    \draw[-] (5,0) -- (5,0.5);
    \draw[-] (9,2.5) -- (5,4.5);
    \draw[-] (1,2.5) -- (5,4.5);
    \draw[-] (9,2.5) -- (5,0.5);
    \draw[-] (1,2.5) -- (5,0.5);
    \draw[-] (6.5,2.5) -- (5,4.5);
    \draw[-] (3.5,2.5) -- (5,4.5);
    \draw[-] (6.5,2.5) -- (5,0.5);
    \draw[-] (3.5,2.5) -- (5,0.5);
    \draw[dashed] (3.5,2.5)--(3.5,0.9);
    \draw[dashed] (6.5,2.5)--(6.5,0.9);
	\draw (10,2.4) node[below] {$x$};
    \draw (4.7,5.2) node[below] {$t$};
    \draw (9,2.4) node[below] {$\delta$};
    \draw (1,2.4) node[below] {$-\delta$};
    \draw (5.25,5) node[below] {$1$};
    \draw (5.4,0.6) node[below] {$-1$};
    \draw (2.9,3.2) node[below] {$T^3$};
    \draw (2.9,2.4) node[below] {$T^2$};
    \draw (3.3,1) node[below] {$\frac{-\tau^+}{\tau^+ +\tau^-}\delta$};
    \draw (6.7,1) node[below] {$\frac{\tau^+}{\tau^+ +\tau^-}\delta$};
    \draw (7.1,3.2) node[below] {$T^1$};
    \draw (7.1,2.4) node[below] {$T^4$};
    \draw (5,2.8) node[below] {$R$};
    \end{tikzpicture}
\end{center}
\caption{Five disjoint domains $T^1,T^2,T^3,T^4,R\subset D$ in Lemma \ref{lem:patching-function}.}
\label{fig3}
\end{figure}
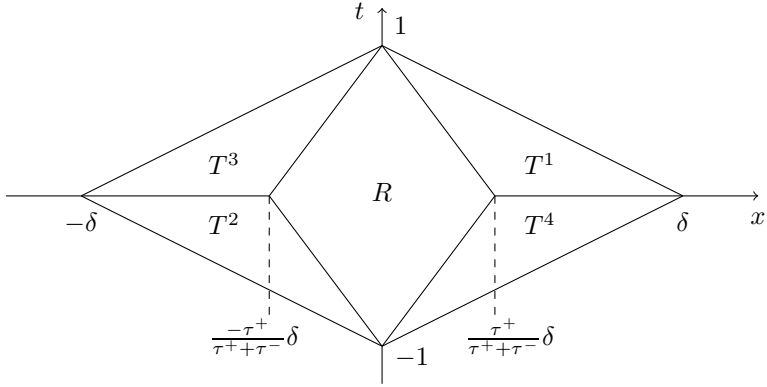


\subsection{Inductive surgeries}\label{subsec:surgeries}

In this subsection, we prove Lemma \ref{lem:main} by utilizing affine maps in Lemma \ref{lem:patching-function}.

Fix any $\epsilon\in(0,1].$ Then we perform the first two surgeries in order that fulfill (a)--(h) for $i=1$ and (a)--(j) for $i=2$. We do not proceed with the $n$th surgery under the assumption that we have performed the surgeries up to the $(n-1)$th one as it is essentially the repetition of the second one under the fulfillment of the first surgery.

\subsubsection{The first surgery}
In this first step, we construct a function $z_1=(v_1,w_1)\in W^{1,\infty}(\Omega_{t_1}^{t_2};\R^2)$, three countable collections $\mathcal{F}_1$ and $\mathcal{F}_1^\pm$, and a constant $c_1>0$ satisfying (a)--(h) for $i=1$, where $c_1$ is independent of the index $i$.

Let $\delta_1\in (0,1]$ be such that
\begin{equation}\label{proof:m-lemma-0-0-0}
\delta_1<\min\Big\{\frac{\epsilon}{2^{1+1}\sqrt{1+L^2}S_M},\frac{\eta_1}{4}\Big\},
\end{equation}
where the constant $S_M>0$ is as in subsection \ref{subsection:completion}.
Since $\nabla z_0$ is uniformly continuous in $Q_1,$ there exists a number $\gamma_1\in(0,\frac{1}{2^{1+1}}]$ such that
\begin{equation}\label{proof:m-lemma-0-0}
(x,t),(y,s)\in Q_1,\,|(x,t)-(y,s)|\le\gamma_1\;\;\Longrightarrow\;\; |\nabla z_0(x,t)-\nabla z_0(y,s)|\le\delta_1.
\end{equation}

Consider the diamond $D_{\delta_1}$. From the Vitali Covering Lemma, we can choose a sequence $\{(x_{1j},t_{1j})\}_{j\in\N}$ in $Q_1$ and a sequence $\{\nu_{1j}\}_{j\in\N}$ in $(0,\gamma_1]$ such that the sequence $\{\tilde{D}_{1j}\}_{j\in\N}$ forms a Vitali cover of $Q_1,$ where $\tilde{D}_{1j}:=(x_{1j},t_{1j})+\nu_{1j} D_{\delta_1}\subset\subset Q_1$ $(j\in\N)$.

For each $j\in\N,$ let
\begin{equation}\label{proof:m-lemma-0-1}
(s_{1j},r_{1j})=((v_0)_x(x_{1j},t_{1j}),(w_0)_t(x_{1j},t_{1j}))\in U^{\lambda_1},
\end{equation}
\begin{equation}\label{proof:m-lemma-0-2}
\tau^+_{1j}=\lambda'_{1}\omega_1(r_{1j})+(1-\lambda'_{1})\omega_2(r_{1j}) -s_{1j}>0,
\end{equation}
\begin{equation}\label{proof:m-lemma-0-3}
\tau^-_{1j}=s_{1j}-(1-\lambda'_{1})\omega_1(r_{1j})-\lambda'_{1}\omega_2(r_{1j})>0,
\end{equation}
and
\begin{equation}\label{proof:m-lemma-0-4}
\varphi_{1j}(x,t)=\left\{
\begin{array}{ll}
  \nu_{1j}\varphi_{\tau^+_{1j},\tau^-_{1j},\delta_1}(\frac{1}{\nu_{1j}}(x-x_{1j},t-t_{1j})), & (x,t)\in \tilde{D}_{1j}, \\
  0, & (x,t)\in\bar{\Omega}_{t_1}^{t_2}\setminus  \tilde{D}_{1j}.
\end{array}
\right.
\end{equation}
Also, for each $j\in\N,$ let $T^{1}_{1j},T^{2}_{1j},T^{3}_{1j},T^{4}_{1j}\subset \tilde{D}_{1j}$ denote the four disjoint triangular domains in each of which $\varphi_{1j}$ is affine and has spatial derivative $\tau^+_{1j}$, and let $R^5_{1j}\subset \tilde{D}_{1j}$ denote the rhombic domain in which $\varphi_{1j}$ is affine and has spatial derivative $-\tau^-_{1j}$; then
\begin{equation}\label{proof:m-lemma-1}
|\tilde{D}_{1j}|=|T^1_{1j}\cup T^2_{1j}\cup T^3_{1j}\cup T^4_{1j}\cup R^5_{1j}|.
\end{equation}
We write
\[
\mathcal{F}_1=\{\tilde{D}_{1j}\,|\,j\in\N\},\;\;\mathcal{F}^+_{1}=\{T^{1}_{1j},T^{2}_{1j},T^{3}_{1j},T^{4}_{1j}\,|\, j\in\N\},\;\;\mbox{and}\;\; \mathcal{F}^-_{1}=\{R^5_{1j}\,|\, j\in\N\};
\]
here, let $\{D_{1j}\}_{j\in\N}$, $\{T_{1j}\}_{j\in\N},$ and $\{R_{1j}\}_{j\in\N}$ be enumerations of $\mathcal{F}_1,$ $\mathcal{F}^+_{1},$ and $\mathcal{F}^-_{1}$, respectively.

Since $\mathcal{F}_1$ is a Vitali cover of $Q_1$, it follows from the definition of $\mathcal{F}^\pm_1$, (\ref{proof:m-lemma-1}), and $2\nu_{1j}\le 2\gamma_1\le \frac{1}{2^1}$ $(j\in\N)$ that (a), (b), (c), (d), (e), and (f) for $i=1$ hold.

To check the rest, define
\[
\varphi_1=\sum_{j=1}^\infty \varphi_{1j}\;\;\mbox{on $\bar{\Omega}_{t_1}^{t_2}$}.
\]
Also, let $N\in\N$ and
\[
\varphi_1^N=\sum_{j=1}^N \varphi_{1j}\;\;\mbox{on $\bar{\Omega}_{t_1}^{t_2}$}.
\]
Then note from the definition of $\varphi_{1j}$ $(j\in\N)$ and (iii) in Lemma \ref{lem:patching-function} that
\begin{equation}\label{finalproof-1}
\begin{split}
|\varphi_{1}^N(x,t)-\varphi_{1}^N(y,s) | & \le |(x,t)-(y,s)|\max_{1\le j\le N}\max\{|(\tau^+_{1j},\delta_1\tau^+_{1j})|,\tau^-_{1j} \} \\
& \le \sqrt{2}S_M|(x,t)-(y,s)|
\end{split}
\end{equation}
for all $(x,t),(y,s)\in \Omega_{t_1}^{t_2}$ and from $\tilde{D}_{1j}\subset\subset Q_1$  $(j\in\N)$ that
\begin{equation}\label{finalproof-2}
\varphi_{1}^N=0\;\;\mbox{and}\;\;\nabla \varphi_{1}^N =0\;\;\mbox{on $\bar{\Omega}_{t_1}^{t_2}\setminus Q_1$}.
\end{equation}
Observe from (v) in Lemma \ref{lem:patching-function} that
\begin{equation}\label{finalproof-3}
\|\varphi_1^N-\varphi_1\|_{L^1(\Omega_{t_1}^{t_2})}\le \sum_{j=N+1}^\infty|\tilde{D}_{1j}| \sup_{j\ge N+1}\frac{\nu_{1j}\tau^+_{1j}\tau^-_{1j}}{\tau^+_{1j}+\tau^-_{1j}}\delta_1 \le  S_M \sum_{j=N+1}^\infty|\tilde{D}_{1j}|\to 0
\end{equation}
as $N\to\infty$; thus,
it follows from (\ref{finalproof-1}) that
\[
|\varphi_1(x,t)-\varphi_1(y,s)|\le \sqrt{2}S_M|(x,t)-(y,s)|\;\;\forall (x,t),(y,s)\in \Omega_{t_1}^{t_2},
\]
that is,
$
\varphi_1\in W^{1,\infty}(\Omega_{t_1}^{t_2}),
$
and from (\ref{finalproof-2}) that
\begin{equation}\label{proof:m-lemma-2}
\varphi_1=0\;\;\mbox{on $\bar{\Omega}_{t_1}^{t_2}\setminus Q_1$}.
\end{equation}
From (iii) in Lemma \ref{lem:patching-function},
\[
\begin{split}
\|\nabla\varphi_1^{N_2}-\nabla\varphi_1^{N_1}\|_{L^1(\Omega_{t_1}^{t_2};\R^2)} & \le \sum_{j=N_1+1}^{N_2}|\tilde{D}_{1j}| \max_{N_1+1\le j\le N_2}\max\{|(\tau^+_{1j},\delta_1\tau^+_{1j})|,\tau^-_{1j} \}\\
& \le \sqrt{2}S_{M} \sum_{j=N_1+1}^{N_2}|\tilde{D}_{1j}|\to 0
\end{split}
\]
as $N_2>N_1\to\infty$; thus, from (\ref{finalproof-3}),
\[
\nabla\varphi_{1}^N\to \nabla\varphi_1\;\;\mbox{in $L^1(\Omega_{t_1}^{t_2};\R^2)$ as $N\to\infty$}
\]
so that (\ref{finalproof-2}) implies that
\begin{equation}\label{proof:m-lemma-2-1}
\nabla\varphi_1=0\;\;\mbox{a.e. in $\Omega_{t_1}^{t_2}\setminus Q_1$.}
\end{equation}

Next, for every $(x,t)\in\bar{\Omega}_{t_1}^{t_2},$ define
\[
\psi_1(x,t)=\int_{0}^x \varphi_1(y,t)\,dy
\]
and
\[
\psi_1^N(x,t)=\int_{0}^x \varphi_1^N(y,t)\,dy\;\;(N\in\N).
\]
Then $\psi_1\in W^{1,\infty}(\Omega_{t_1}^{t_2})$,
\begin{equation}\label{finalproof-7}
(\psi_1)_x=\varphi_1\;\;\mbox{in $\Omega_{t_1}^{t_2}$,}
\end{equation}
from (iv) and (v) in Lemma \ref{lem:patching-function},
\begin{equation}\label{finalproof-4}
\begin{split}
\|\psi_1^N-\psi_1\|_{L^1(\Omega_{t_1}^{t_2})} & =\int_{\Omega_{t_1}^{t_2}} \bigg|\int_0^x (\varphi_1^N(y,t)-\varphi_1(y,t))\,dy\bigg|\,dxdt \\
& = \sum_{j=N+1}^{\infty}\int_{\tilde{D}_{1j}} \bigg|\int_0^x \varphi_{1j}(y,t)\,dy\bigg|\,dxdt \\
&\le \sum_{j=N+1}^{\infty} |\tilde{D}_{1j}|\frac{\nu_{1j}^2\delta_1^2\tau^+_{1j}\tau^-_{1j}}{2(\tau^+_{1j}+\tau^-_{1j})} \le S_M\sum_{j=N+1}^{\infty} |\tilde{D}_{1j}| \to 0
\end{split}
\end{equation}
as $N\to\infty$,
and from (iii) in Lemma \ref{lem:patching-function},
\[
\begin{split}
\|\nabla\psi_1^{N_2}-\nabla\psi_1^{N_1} & \|_{L^1(\Omega_{t_1}^{t_2};\R^2)}  =\int_{\Omega_{t_1}^{t_2}} \bigg|\int_0^x (\nabla\varphi_1^{N_2}(y,t)-\nabla\varphi_1^{N_1}(y,t))\,dy\bigg|\,dxdt \\
& = \sum_{j=N_1+1}^{N_2} \int_{\tilde{D}_{1j}} \bigg|\int_0^x \nabla\varphi_{1j}(y,t)\,dy\bigg|\,dxdt 
\end{split}
\]
\[
\begin{split}
&\le \max_{N_1+1\le j\le N_2} 2\nu_{1j}\delta_1\max\{|(\tau^+_{1j},\delta_1\tau^+_{1j})|,\tau^-_{1j}\} \sum_{j=N_1+1}^{N_2} |\tilde{D}_{1j}|\\
&\le S_M\sum_{j=N_1+1}^{N_2} |\tilde{D}_{1j}| \to 0
\end{split}
\]
as $N_2>N_1\to\infty$; thus,
\begin{equation}\label{finalproof-5}
\nabla\psi_{1}^N\to \nabla\psi_1\;\;\mbox{in $L^1(\Omega_{t_1}^{t_2};\R^2)$ as $N\to\infty$}.
\end{equation}
Note from the definition of $\psi^N_1$ $(N\in\N)$ and (iii) and (iv) in Lemma \ref{lem:patching-function} that for each $N\in\N$,
\begin{equation*}
\psi_{1}^N=0\;\;\mbox{and}\;\;\nabla \psi_{1}^N =0\;\;\mbox{on $\bar{\Omega}_{t_1}^{t_2}\setminus Q_1$};
\end{equation*}
thus, letting $N\to\infty$, it follows from (\ref{finalproof-4}) and (\ref{finalproof-5}) that
\begin{equation}\label{finalproof-6}
\psi_1=0\;\;\mbox{on $\bar{\Omega}_{t_1}^{t_2}\setminus Q_1$}\;\;\mbox{and}\;\; \nabla\psi_1=0\;\;\mbox{a.e. in $\Omega_{t_1}^{t_2}\setminus Q_1$}.
\end{equation}

In turn, define
\begin{equation}\label{proof:m-lemma-3-2}
z_1=(v_1,w_1)=z_0+(\varphi_1,\psi_1)\;\;\mbox{on $\bar{\Omega}_{t_1}^{t_2}$};
\end{equation}
then $z_1\in W^{1,\infty}(\Omega_{t_1}^{t_2};\R^2)$ as $(\varphi_1,\psi_1)\in W^{1,\infty}(\Omega_{t_1}^{t_2};\R^2)$.
From (\ref{proof:m-lemma-2}) and (\ref{finalproof-6}), we have
\[
z_1=z_0\;\;\mbox{on $\bar{\Omega}_{t_1}^{t_2}\setminus Q_1$;}
\]
hence, the third of (g)  for $i=1$ holds.
Also, the fourth of (g)  for $i=1$ follows from (\ref{proof:m-lemma-2-1}), (\ref{finalproof-6}), and $Q_1\subset Q$.
The first of (h) for $i=1$ is a consequence of (\ref{subsolution-1}) and (\ref{finalproof-7}).

From (\ref{proof:m-lemma-0-0}), (\ref{proof:m-lemma-0-1}), (\ref{proof:m-lemma-0-2}), (\ref{proof:m-lemma-0-3}), (\ref{proof:m-lemma-0-4}), (\ref{proof:m-lemma-3-2}), the definition of $U^\pm_1$, $\lambda_1'$, $\eta_1$, and $\mathcal{F}^\pm_1$, and Lemma \ref{lem:patching-function}, the first and second of (g) for $i=1$ are fulfilled.

From (\ref{proof:m-lemma-0-0-0}), (\ref{proof:m-lemma-3-2}), and (iii) in Lemma \ref{lem:patching-function},
we have
\begin{equation}\label{proof:m-lemma-4}
\|(v_1)_t-(v_0)_t\|_{L^\infty(\Omega_{t_1}^{t_2} )}= \|(\varphi_1)_t\|_{L^\infty(\Omega_{t_1}^{t_2} )} \le \delta_1\sup_{j\in\N}\tau^+_{1j} \le \delta_1 S_M\le \frac{\epsilon}{2^{1+1}};
\end{equation}
hence, the second of (h) for $i=1$  holds.

From (\ref{proof:m-lemma-0-0-0}), (\ref{proof:m-lemma-0-4}), (\ref{proof:m-lemma-3-2}),  $\nu_{1j}\le\gamma_1\le\frac{1}{2^{1+1}}$ $(j\in\N)$, and (v) in Lemma \ref{lem:patching-function}, we have
\[
\begin{split}
\|z_1 -z_0\|_{L^\infty(\Omega_{t_1}^{t_2};\R^2)} & = \|(\varphi_1,\psi_1)\|_{L^\infty(\Omega_{t_1}^{t_2};\R^2)} \\
& \le (1+L^2)^{\frac{1}{2}}\delta_1\sup_{j\in\N}\frac{\nu_{1j}\tau^+_{1j}\tau^-_{1j}}{\tau^+_{1j}+\tau^-_{1j}} \le(1+L^2)^{\frac{1}{2}}\delta_1 S_M \le \frac{\epsilon}{2^{1+1}};
\end{split}
\]
thus, the third of (h) for $i=1$  is fulfilled.

Finally, to check the fourth of (h) for $i=1$, note  from the first and third of (h) for $i=1$, (\ref{proof:m-lemma-2-1}), (\ref{finalproof-6}), (\ref{proof:m-lemma-3-2}), and (\ref{proof:m-lemma-4}) that
\[
\begin{split}
\|\nabla z_1  & \|_{L^\infty(\Omega_{t_1}^{t_2};\M^{2\times 2})} \le \|((v_1)_x,(w_1)_t)\|_{L^\infty(\Omega_{t_1}^{t_2};\R^{2})}+ \|(v_1)_t\|_{L^\infty(\Omega_{t_1}^{t_2})} + \|(w_1)_x\|_{L^\infty(\Omega_{t_1}^{t_2})} \\
\le & \,\|((v_0)_x,(w_0)_t)\|_{L^\infty(\Omega_{t_1}^{t_2}\setminus Q_1;\R^{2})} + \|((v_1)_x,(w_1)_t)\|_{L^\infty(Q_1;\R^{2})}+\|(v_0)_t\|_{L^\infty(\Omega_{t_1}^{t_2})}  \\
& + \|(\varphi_1)_t\|_{L^\infty(\Omega_{t_1}^{t_2})} +\|v_1\|_{L^\infty(\Omega_{t_1}^{t_2})}  \\
\le & \, 2\|\nabla z_0\|_{L^\infty(\Omega_{t_1}^{t_2};\M^{2\times 2})} +C_U+\|z_0\|_{L^\infty(\Omega_{t_1}^{t_2};\R^2)}+\frac{2}{2^{1+1}}\le c_1
\end{split}
\]
where $C_U:=\sup_{(s,r)\in U}|(s,r)|$ and
\[
c_1:=2\|\nabla z_0\|_{L^\infty(\Omega_{t_1}^{t_2};\M^{2\times 2})} +C_U+\|z_0\|_{L^\infty(\Omega_{t_1}^{t_2};\R^2)}+1>0.
\]
Therefore, the fourth of (h) for $i=1$  follows.

The first step is now finished.

\underline{\textbf{The second surgery:}}
In the second step, we construct a function $z_2=(v_2,w_2)\in W^{1,\infty}(\Omega_{t_1}^{t_2};\R^2)$, three countable collections $\mathcal{F}_2$ and $\mathcal{F}_2^\pm$, and two positive constants $c_2$ and $c_3$ satisfying (a)--(j) for $i=2$, where the constant $c_1>0$ is as in the first step, and the two constants $c_2$ and $c_3$ are independent of the index $i$.

Let $\delta_2\in (0,1]$ be such that
\begin{equation}\label{proof:s-lemma-0-0-0}
\delta_2<\min\Big\{\frac{\epsilon}{2^{2+1}\sqrt{1+L^2}S_M},\frac{\eta_2}{4}\Big\}.
\end{equation}
Since $\nabla z_1$ is uniformly continuous in $Q_2,$ there exists a number $\gamma_2\in(0,\frac{1}{2^{2+1}}]$ such that
\begin{equation}\label{proof:s-lemma-0-0-1}
(x,t),(y,s)\in Q_2,\,|(x,t)-(y,s)|\le\gamma_2\;\;\Longrightarrow\;\; |\nabla z_1(x,t)-\nabla z_1(y,s)|\le\delta_2.
\end{equation}

For each $j\in\N,$ since $\nabla z_1$ is uniformly continuous in $T_{1j}$ and in $R_{1j}$,  there exists a number $\gamma_{1j}\in(0,\frac{1}{2^{2+1}}]$ such that
\begin{equation}\label{proof:s-lemma-0-0-2}
(x,t),(y,s)\in T_{1j},\,|(x,t)-(y,s)|\le\gamma_{1j}\;\;\Longrightarrow\;\; |\nabla z_1(x,t)-\nabla z_1(y,s)|\le\delta_2
\end{equation}
and that
\begin{equation}\label{proof:s-lemma-0-0-3}
(x,t),(y,s)\in R_{1j},\,|(x,t)-(y,s)|\le\gamma_{1j}\;\;\Longrightarrow\;\; |\nabla z_1(x,t)-\nabla z_1(y,s)|\le\delta_2.
\end{equation}

Consider the diamond $D_{\delta_2}$. From the Vitali Covering Lemma, we can choose a sequence $\{(x_{2j},t_{2j})\}_{j\in\N}$ in $Q_2$ and a sequence $\{\nu_{2j}\}_{j\in\N}$ in $(0,\gamma_2]$ such that the sequence $\{\tilde{D}_{2j}\}_{j\in\N}$ forms a Vitali cover of $Q_2,$ where $\tilde{D}_{2j}:=(x_{2j},t_{2j})+\nu_{2j} D_{\delta_2}\subset\subset Q_2$ $(j\in\N)$.

Let $j\in\N.$ Also, from the Vitali Covering Lemma, we can choose a sequence $\{(x^+_{1jk},t^+_{1jk})\}_{k\in\N}$ in $T_{1j}$ and a sequence $\{\nu^+_{1jk}\}_{k\in\N}$ in $(0,\gamma_{1j}]$ such that the sequence $\{\tilde{D}^+_{1jk}\}_{k\in\N}$ forms a Vitali cover of $T_{1j}$, where $\tilde{D}^+_{1jk}:=(x^+_{1jk},t^+_{1jk})+\nu^+_{1jk} D_{\delta_2}\subset\subset T_{1j}$ $(k\in\N)$.
Likewise, we can choose a sequence $\{(x^-_{1jk},t^-_{1jk})\}_{k\in\N}$ in $R_{1j}$ and a sequence $\{\nu^-_{1jk}\}_{k\in\N}$ in $(0,\min\{\gamma_{1j},\lambda_1 S_M\}]$ such that the sequence $\{\tilde{D}^-_{1jk}\}_{k\in\N}$ forms a Vitali cover of $R_{1j}$, where $\tilde{D}^-_{1jk}:=(x^-_{1jk},t^-_{1jk})+\nu^-_{1jk} D_{\delta_2}\subset\subset R_{1j}$ $(k\in\N)$.

For each $j\in\N,$ let
\begin{equation}\label{proof:s-lemma-0-1-0}
(s_{2j},r_{2j})=((v_1)_x(x_{2j},t_{2j}),(w_1)_t(x_{2j},t_{2j}))\in U^{\lambda_2}\setminus\bar{U}^{\lambda_1},
\end{equation}
\begin{equation}\label{proof:s-lemma-0-1-1}
\tau^+_{2j}=\lambda'_{2}\omega_1(r_{2j})+(1-\lambda'_{2})\omega_2(r_{2j}) -s_{2j}>0,
\end{equation}
\begin{equation}\label{proof:s-lemma-0-1-2}
\tau^-_{2j}=s_{2j}-(1-\lambda'_{2})\omega_1(r_{2j})-\lambda'_{2}\omega_2(r_{2j})>0,
\end{equation}
and
\begin{equation}\label{proof:s-lemma-0-1-3}
\varphi_{2j}(x,t)=\left\{
\begin{array}{ll}
  \nu_{2j}\varphi_{\tau^+_{2j},\tau^-_{2j},\delta_2}(\frac{1}{\nu_{2j}}(x-x_{2j},t-t_{2j})), & (x,t)\in \tilde{D}_{2j}, \\
  0, & (x,t)\in\bar{\Omega}_{t_1}^{t_2}\setminus  \tilde{D}_{2j}.
\end{array}
\right.
\end{equation}
Also, for each $j\in\N,$ let $T^{1}_{2j},T^{2}_{2j},T^{3}_{2j},T^{4}_{2j}\subset \tilde{D}_{2j}$ denote the four disjoint triangular domains in each of which $\varphi_{2j}$ is affine and has spatial derivative $\tau^+_{2j}$, and let $R^5_{2j}\subset \tilde{D}_{2j}$ denote the rhombic domain in which $\varphi_{2j}$ is affine and has spatial derivative $-\tau^-_{2j}$; then
\begin{equation}\label{proof:s-lemma-0-1-4}
|\tilde{D}_{2j}|=|T^1_{2j}\cup T^2_{2j}\cup T^3_{2j}\cup T^4_{2j}\cup R^5_{2j}|.
\end{equation}

Let $j,k\in\N,$ and let
\begin{equation}\label{proof:s-lemma-0-2-0}
(s^+_{1jk},r^+_{1jk})=((v_1)_x(x^+_{1jk},t^+_{1jk}),(w_1)_t(x^+_{1jk},t^+_{1jk}))\in U^+_1,
\end{equation}
\begin{equation}\label{proof:s-lemma-0-2-1}
\tau^{++}_{1jk}=\lambda'_{2}\omega_1(r^+_{1jk})+(1-\lambda'_{2})\omega_2(r^+_{1jk}) -s^+_{1jk}>0,
\end{equation}
\begin{equation}\label{proof:s-lemma-0-2-2}
\tau^{+-}_{1jk}=s^+_{1jk}-(1-\lambda'_{2})\omega_1(r^+_{1jk})-\lambda'_{2}\omega_2(r^+_{1jk})>0,
\end{equation}
and
\begin{equation}\label{proof:s-lemma-0-2-3}
\varphi^+_{1jk}(x,t)=\left\{
\begin{array}{ll}
  \nu^+_{1jk}\varphi_{\tau^{++}_{1jk},\tau^{+-}_{1jk},\delta_2}(\frac{1}{\nu^+_{1jk}}(x-x^+_{1jk},t-t^+_{1jk})), & (x,t)\in \tilde{D}^+_{1jk}, \\
  0, & (x,t)\in\bar{\Omega}_{t_1}^{t_2}\setminus  \tilde{D}^+_{1jk}.
\end{array}
\right.
\end{equation}
Observe from (\ref{proof:s-lemma-0-2-0}) that
\begin{equation}\label{finalproof-9}
\begin{split}
\frac{\tau^{++}_{1jk}}{\tau^{++}_{1jk}+\tau^{+-}_{1jk}} & = \frac{\lambda'_{2}\omega_1(r^+_{1jk})+(1-\lambda'_{2})\omega_2(r^+_{1jk}) -s^+_{1jk}}{(1-2\lambda'_{2})(\omega_2(r^+_{1jk})-\omega_1(r^+_{1jk}))} \\
& \le \frac{(\lambda_1-\lambda'_{2})(\omega_2(r^+_{1jk})-\omega_1(r^+_{1jk}))}{(1-2\lambda'_{2})(\omega_2(r^+_{1jk})-\omega_1(r^+_{1jk}))} = \frac{\lambda_1-\lambda_2'}{1-2\lambda_2'}
\end{split}
\end{equation}
and that
\begin{equation}\label{finalproof-12}
\frac{\tau^{++}_{1jk}}{\tau^{++}_{1jk}+\tau^{+-}_{1jk}}\ge \frac{\lambda_2-\lambda_2'}{1-2\lambda_2'}.
\end{equation}
Note also that
\begin{equation}\label{finalproof-14}
\begin{split}
\frac{\tau^{+-}_{1jk}}{\tau^{++}_{1jk}+\tau^{+-}_{1jk}} & = \frac{s^+_{1jk}-(1-\lambda'_{2})\omega_1(r^+_{1jk})-\lambda'_{2}\omega_2(r^+_{1jk})}{(1-2\lambda'_{2})(\omega_2(r^+_{1jk})-\omega_1(r^+_{1jk}))} \\
& \ge \frac{(1-\lambda_1-\lambda'_{2})(\omega_2(r^+_{1jk})-\omega_1(r^+_{1jk}))}{(1-2\lambda'_{2})(\omega_2(r^+_{1jk})-\omega_1(r^+_{1jk}))} = \frac{1-\lambda_1-\lambda_2'}{1-2\lambda_2'}.
\end{split}
\end{equation}
Here, we claim that
\begin{equation}\label{finalproof-15}
\frac{1-\lambda_1-\lambda_2'}{1-2\lambda_2'}> 1-(\beta_2-\beta_1).
\end{equation}
To check this, we start from the inequality $\lambda_1>\lambda_2'$. Since $\ell_0<\kappa_0\ell_0<1,$ we have $\frac{1}{\kappa_0\ell_0}-1>0$ so that
\[
\bigg(\frac{1}{\kappa_0\ell_0}-1\bigg)\lambda_1>\bigg(\frac{1}{\kappa_0\ell_0}-1\bigg)\lambda_2';
\]
that is,
\[
1-\lambda_1-\lambda_2'>1-2\lambda_2'-\frac{1}{\kappa_0\ell_0}(\lambda_1-\lambda_2').
\]
As $1-2\lambda_2'>0$, we now have
\[
\frac{1-\lambda_1-\lambda_2'}{1-2\lambda_2'}>1-\frac{1}{\kappa_0\ell_0}\cdot\frac{\lambda_1-\lambda_2'}{1-2\lambda_2'}=1-\frac{\alpha_1}{\kappa_0\ell_0}=1-(\beta_2-\beta_1);
\]
hence the claim holds.

Let $T^{+,1}_{1jk},T^{+,2}_{1jk},T^{+,3}_{1jk},T^{+,4}_{1jk}\subset \tilde{D}^{+}_{1jk}$ denote the four disjoint triangular domains in each of which $\varphi^{+}_{1jk}$ is affine and has spatial derivative $\tau^{++}_{1jk}$, and let $R^{+,5}_{1jk}\subset \tilde{D}^{+}_{1jk}$ denote the rhombic domain in which $\varphi^{+}_{1jk}$ is affine and has spatial derivative $-\tau^{+-}_{1jk}$; then
\begin{equation}\label{proof:s-lemma-0-2-4}
|\tilde{D}^{+}_{1jk}|=|T^{+,1}_{1jk}\cup T^{+,2}_{1jk}\cup T^{+,3}_{1jk}\cup T^{+,4}_{1jk}\cup R^{+,5}_{1jk}|.
\end{equation}

Let $j,k\in\N,$ and let
\begin{equation}\label{proof:s-lemma-0-3-0}
(s^-_{1jk},r^-_{1jk})=((v_1)_x(x^-_{1jk},t^-_{1jk}),(w_1)_t(x^-_{1jk},t^-_{1jk}))\in U^-_1,
\end{equation}
\begin{equation}\label{proof:s-lemma-0-3-1}
\tau^{-+}_{1jk}=\lambda'_{2}\omega_1(r^-_{1jk})+(1-\lambda'_{2})\omega_2(r^-_{1jk}) -s^-_{1jk}>0,
\end{equation}
\begin{equation}\label{proof:s-lemma-0-3-2}
\tau^{--}_{1jk}=s^-_{1jk}-(1-\lambda'_{2})\omega_1(r^-_{1jk})-\lambda'_{2}\omega_2(r^-_{1jk})>0,
\end{equation}
and
\begin{equation}\label{proof:s-lemma-0-3-3}
\varphi^-_{1jk}(x,t)=\left\{
\begin{array}{ll}
  \nu^-_{1jk}\varphi_{\tau^{-+}_{1jk},\tau^{--}_{1jk},\delta_2}(\frac{1}{\nu^-_{1jk}}(x-x^-_{1jk},t-t^-_{1jk})), & (x,t)\in \tilde{D}^-_{1jk}, \\
  0, & (x,t)\in\bar{\Omega}_{t_1}^{t_2}\setminus  \tilde{D}^-_{1jk}.
\end{array}
\right.
\end{equation}
Observe from (\ref{proof:s-lemma-0-3-0}) that
\begin{equation}\label{finalproof-10}
\begin{split}
\frac{\tau^{--}_{1jk}}{\tau^{-+}_{1jk}+\tau^{--}_{1jk}} & = \frac{s^-_{1jk}-(1-\lambda'_{2})\omega_1(r^-_{1jk})-\lambda'_{2}\omega_2(r^-_{1jk})}{(1-2\lambda'_{2})(\omega_2(r^-_{1jk})-\omega_1(r^-_{1jk}))} \\
& \le \frac{(\lambda_1-\lambda'_{2})(\omega_2(r^-_{1jk})-\omega_1(r^-_{1jk}))}{(1-2\lambda'_{2})(\omega_2(r^-_{1jk})-\omega_1(r^-_{1jk}))} = \frac{\lambda_1-\lambda_2'}{1-2\lambda_2'}
\end{split}
\end{equation}
and that
\begin{equation}\label{finalproof-13}
\frac{\tau^{--}_{1jk}}{\tau^{-+}_{1jk}+\tau^{--}_{1jk}}\ge \frac{\lambda_2-\lambda_2'}{1-2\lambda_2'}.
\end{equation}
As above, we also have
\[
\frac{\tau_{1jk}^{-+}}{\tau_{1jk}^{-+}+\tau_{1jk}^{--}}\ge\frac{1-\lambda_1-\lambda_2'}{1-2\lambda_2'}>1-(\beta_2-\beta_1).
\]

Let $T^{-,1}_{1jk},T^{-,2}_{1jk},T^{-,3}_{1jk},T^{-,4}_{1jk}\subset \tilde{D}^{-}_{1jk}$ denote the four disjoint triangular domains in each of which $\varphi^{-}_{1jk}$ is affine and has spatial derivative $\tau^{-+}_{1jk}$, and let $R^{-,5}_{1jk}\subset \tilde{D}^{-}_{1jk}$ denote the rhombic domain in which $\varphi^{-}_{1jk}$ is affine and has spatial derivative $-\tau^{--}_{1jk}$; then
\begin{equation}\label{proof:s-lemma-0-3-4}
|\tilde{D}^{-}_{1jk}|=|T^{-,1}_{1jk}\cup T^{-,2}_{1jk}\cup T^{-,3}_{1jk}\cup T^{-,4}_{1jk}\cup R^{-,5}_{1jk}|.
\end{equation}

We write
\[
\mathcal{F}_2= \{\tilde{D}_{2j}\,|\, j\in\N\}\cup \{\tilde{D}^+_{1jk}\,|\, j,k\in\N\}\cup \{\tilde{D}^-_{1jk}\,|\, j,k\in\N\},
\]
\[
\begin{split}
\mathcal{F}^+_{2}= & \{T^{1}_{2j},T^{2}_{2j},T^{3}_{2j},T^{4}_{2j}\,|\, j\in\N\}\cup \{T^{+,1}_{1jk},T^{+,2}_{1jk},T^{+,3}_{1jk},T^{+,4}_{1jk}\,|\, j,k\in\N\}\\
& \cup \{T^{-,1}_{1jk},T^{-,2}_{1jk},T^{-,3}_{1jk},T^{-,4}_{1jk}\,|\, j,k\in\N\},
\end{split}
\]
and
\[
\mathcal{F}^-_{2}=\{R^5_{2j}\,|\, j\in\N\}\cup \{R^{+,5}_{1jk}\,|\, j,k\in\N\}\cup \{R^{-,5}_{1jk}\,|\, j,k\in\N\};
\]
here, let $\{D_{2j}\}_{j\in\N},$ $\{T_{2j}\}_{j\in\N},$ and $\{R_{2j}\}_{j\in\N}$ be enumerations of $\mathcal{F}_{2}$, $\mathcal{F}^+_{2},$ and $\mathcal{F}^-_{2}$, respectively.

Since $\mathcal{F}_2$ is a Vitali cover of $Q_1\cup Q_2,$ it follows from the definition of $\mathcal{F}^\pm_2,$ (\ref{proof:s-lemma-0-1-4}), (\ref{proof:s-lemma-0-2-4}), (\ref{proof:s-lemma-0-3-4}), $2\nu_{2j}\le 2\gamma_2\le\frac{1}{2^2}$ $(j\in\N)$, and $2\nu^\pm_{1jk}\le 2\gamma_{1j}\le\frac{1}{2^2}$ $(j,k\in\N)$ that (a), (b), (c), (d), (e), and (f) for $i=2$ hold.

To check the rest, define
\begin{equation}\label{finalproof-8}
\varphi_2=\sum_{j=1}^\infty \varphi_{2j}+ \sum_{j=1}^\infty\sum_{k=1}^\infty \varphi^+_{1jk} + \sum_{j=1}^\infty\sum_{k=1}^\infty \varphi^-_{1jk}\;\;\mbox{on $\bar{\Omega}_{t_1}^{t_2}$};
\end{equation}
then as in the first step, we can see that
\[
|\varphi_2(x,t)-\varphi_2(y,s)|\le \sqrt{2} S_M|(x,t)-(y,s)|\;\;\forall (x,t),(y,s)\in\Omega_{t_1}^{t_2},
\]
that is, $\varphi_2\in W^{1,\infty}(\Omega_{t_1}^{t_2})$,
\begin{equation}\label{proof:s-lemma-0-4-0}
\varphi_2=0\;\;\mbox{on $\bar{\Omega}_{t_1}^{t_2}\setminus(Q_1\cup Q_2)$},
\end{equation}
and
\begin{equation}\label{proof:s-lemma-0-4-1}
\nabla\varphi_2=0\;\;\mbox{a.e. in $\Omega_{t_1}^{t_2}\setminus(Q_1\cup Q_2).$}
\end{equation}

Next, for every $(x,t)\in\bar{\Omega}_{t_1}^{t_2},$ define
\[
\psi_2(x,t)=\int_{0}^x \varphi_2(y,t)\,dy;
\]
then as in the first step, we can check that $\psi_2\in W^{1,\infty}(\Omega_{t_1}^{t_2})$,
\begin{equation}\label{proof:s-lemma-0-5-1}
(\psi_2)_x=\varphi_2\;\;\mbox{in $\Omega_{t_2}^{t_1}$,}
\end{equation}
\begin{equation}\label{proof:s-lemma-0-5-0}
\psi_2=0\;\;\mbox{on $\bar{\Omega}_{t_1}^{t_2}\setminus(Q_1\cup Q_2)$,}\;\;\mbox{and}\;\; \nabla\psi_2=0\;\;\mbox{a.e. in $\Omega_{t_1}^{t_2}\setminus(Q_1\cup Q_2)$.}
\end{equation}

In turn, define
\begin{equation}\label{proof:s-lemma-0-6-0}
z_2=(v_2,w_2)=z_1+(\varphi_2,\psi_2)\;\;\mbox{on $\bar{\Omega}_{t_1}^{t_2}$};
\end{equation}
then $z_2\in W^{1,\infty}(\Omega_{t_1}^{t_2};\R^2)$ as $(\varphi_2,\psi_2)\in W^{1,\infty}(\Omega_{t_1}^{t_2};\R^2)$.
From (\ref{proof:s-lemma-0-4-0}), (\ref{proof:s-lemma-0-5-0}), and the third of (g) for $i=1$, we have
\[
z_2=z_0\;\;\mbox{on $\bar{\Omega}_{t_1}^{t_2}\setminus (Q_1\cup Q_2)$;}
\]
hence, the third of (g) for $i=2$ holds. Also, the fourth of (g) for $i=2$ follows from (\ref{proof:s-lemma-0-4-1}), (\ref{proof:s-lemma-0-5-0}), and $Q_1\cup Q_2\subset Q$. The first of (h) for $i=2$ is implied by (\ref{proof:s-lemma-0-5-1}) and the first of (h) for $i=1$.

From (\ref{proof:s-lemma-0-0-1}), (\ref{proof:s-lemma-0-0-2}), (\ref{proof:s-lemma-0-0-3}), (\ref{proof:s-lemma-0-1-0}), (\ref{proof:s-lemma-0-1-1}), (\ref{proof:s-lemma-0-1-2}), (\ref{proof:s-lemma-0-1-3}), (\ref{proof:s-lemma-0-2-0}), (\ref{proof:s-lemma-0-2-1}), (\ref{proof:s-lemma-0-2-2}), (\ref{proof:s-lemma-0-2-3}), (\ref{proof:s-lemma-0-3-0}), (\ref{proof:s-lemma-0-3-1}), (\ref{proof:s-lemma-0-3-2}), (\ref{proof:s-lemma-0-3-3}), (\ref{proof:s-lemma-0-6-0}), the definition of $U^\pm_2$, $\lambda_2'$, $\eta_2$, and $\mathcal{F}^\pm_2$, and Lemma \ref{lem:patching-function}, the first and second of (g) for $i=2$ are satisfied.

From (\ref{proof:s-lemma-0-0-0}), (\ref{proof:s-lemma-0-6-0}), and (iii) in Lemma \ref{lem:patching-function}, we have
\begin{equation}\label{proof:s-lemma-0-7-0}
\begin{split}
\|(v_2)_t-(v_1)_t\|_{L^\infty(\Omega_{t_1}^{t_2})}& = \|(\varphi_2)_t\|_{L^\infty(\Omega_{t_1}^{t_2})} \\
& \le\delta_2\sup_{j,k\in\N} \max\{\tau^{+}_{2j},\tau^{++}_{1jk},\tau^{-+}_{1jk}\}\le\delta_2 S_M\le\frac{\epsilon}{2^{2+1}};
\end{split}
\end{equation}
hence, the second of (h) for $i=2$ holds.

From (\ref{proof:s-lemma-0-0-0}), (\ref{proof:s-lemma-0-1-3}), (\ref{proof:s-lemma-0-2-3}), (\ref{proof:s-lemma-0-3-3}), (\ref{proof:s-lemma-0-6-0}), $\nu_{2j}\le\gamma_2\le\frac{1}{2^{2+1}}$ $(j\in\N)$, and $\nu^\pm_{1jk}\le\gamma_{1j}\le\frac{1}{2^{2+1}}$ $(j,k\in\N)$, we have
\[
\begin{split}
\|z_2- & z_1\|_{L^\infty(\Omega_{t_1}^{t_2};\R^2)}= \|(\varphi_2,\psi_2)\|_{L^\infty(\Omega_{t_1}^{t_2};\R^2)} \\
& \le (1+L^2)^{\frac{1}{2}}\delta_2 \sup_{j,k\in\N}\max\bigg\{\frac{\nu_{2j}\tau^+_{2j}\tau^-_{2j}}{\tau^+_{2j}+\tau^-_{2j}}, \frac{\nu^+_{1jk}\tau^{++}_{1jk}\tau^{+-}_{1jk}}{\tau^{++}_{1jk}+\tau^{+-}_{1jk}}, \frac{\nu^-_{1jk}\tau^{-+}_{1jk}\tau^{--}_{1jk}}{\tau^{-+}_{1jk}+\tau^{--}_{1jk}} \bigg\} \\
& \le (1+L^2)^{\frac{1}{2}}\delta_2 S_M \le \frac{\epsilon}{2^{2+1}};
\end{split}
\]
thus, the third of (h) for $i=2$ is satisfied.

To check the fourth of (h) for $i=2$, note from the first and third of (h) for $i=2$, (\ref{proof:m-lemma-4}), (\ref{proof:s-lemma-0-4-1}), (\ref{proof:s-lemma-0-5-0}), (\ref{proof:s-lemma-0-6-0}), and (\ref{proof:s-lemma-0-7-0}) that

\[
\begin{split}
\|\nabla z_2 & \|_{L^\infty(\Omega_{t_1}^{t_2};\M^{2\times 2})} \le \|((v_2)_x,(w_2)_t)\|_{L^\infty(\Omega_{t_1}^{t_2};\R^2)}+ \|(v_2)_t\|_{L^\infty(\Omega_{t_1}^{t_2})} \\
&\quad\quad\quad\quad\quad\quad\quad\; + \|(w_2)_x\|_{L^\infty(\Omega_{t_1}^{t_2})} \\
\le & \, \|((v_0)_x,(w_0)_t)\|_{L^\infty(\Omega_{t_1}^{t_2}\setminus (Q_1\cup Q_2);\R^2)}+ \|((v_2)_x,(w_2)_t)\|_{L^\infty(Q_1\cup Q_2;\R^2)} \\
&\, + \|(v_0)_t\|_{L^\infty(\Omega_{t_1}^{t_2})} + \|(\varphi_1)_t\|_{L^\infty(\Omega_{t_1}^{t_2})} +\|(\varphi_2)_t\|_{L^\infty(\Omega_{t_1}^{t_2})}  + \|v_2\|_{L^\infty(\Omega_{t_1}^{t_2})} \\
\le & \,2\|\nabla z_0\|_{L^\infty(\Omega_{t_1}^{t_2};\M^{2\times 2})}+ C_U+\|z_0\|_{L^\infty(\Omega_{t_1}^{t_2};\R^2)} + 2\Big(\frac{1}{2^{1+1}}+\frac{1}{2^{2+1}}\Big)\le c_1.
\end{split}
\]
Hence, the fourth of (h) for $i=2$ is true.

We now verify (i) for $i=2.$ Note from (\ref{finalproof-8}), (\ref{proof:s-lemma-0-4-1}), (\ref{proof:s-lemma-0-5-0}), and (\ref{proof:s-lemma-0-6-0}) that
\[
\begin{split}
\int_{\Omega_{t_1}^{t_2}}| & \nabla z_2-\nabla z_1|\,dxdt=\int_{\Omega_{t_1}^{t_2}} |\nabla(\varphi_2,\psi_2)|\,dxdt =\bigg(\int_{Q_1}+\int_{Q_2}\bigg) |\nabla(\varphi_2,\psi_2)|\,dxdt \\
=& \bigcup_{j,k\in\N}\int_{\tilde{D}^+_{1jk}}|\nabla(\varphi^+_{1jk},\psi_2)|\,dxdt +\bigcup_{j,k\in\N}\int_{\tilde{D}^-_{1jk}}|\nabla(\varphi^-_{1jk},\psi_2)|\,dxdt \\
& +\bigcup_{j\in\N}\int_{\tilde{D}_{2j}}|\nabla(\varphi_{2j},\psi_2)|\,dxdt.
\end{split}
\]
Let $j,k\in\N.$ Then from (\ref{choice-lamda-i}), (\ref{choice-beta_i}), (\ref{proof:s-lemma-0-2-3}), (\ref{finalproof-9}), (\ref{proof:s-lemma-0-5-1}), and (iii), (v), and (vi) in Lemma \ref{lem:patching-function},
\[
\begin{split}
\int_{\tilde{D}^+_{1jk}} & |\nabla(\varphi^+_{1jk},\psi_2)|\,dxdt = \int_{T^{+,1}_{1jk}\cup T^{+,2}_{1jk}\cup T^{+,3}_{1jk}\cup T^{+,4}_{1jk}\cup R^{+,5}_{1jk}}|\nabla(\varphi^+_{1jk},\psi_2)|\,dxdt \\
= & \int_{T^{+,1}_{1jk}\cup T^{+,2}_{1jk}\cup T^{+,3}_{1jk}\cup T^{+,4}_{1jk}}((1+\delta_2^2)(\tau^{++}_{1jk})^2+(\varphi^+_{1jk})^2+((\psi_2)_t)^2)^{\frac{1}{2}}\,dxdt\\
& + \int_{R^{+,5}_{1jk}}((\tau^{+-}_{1jk})^2 +(\varphi^+_{1jk})^2+((\psi_2)_t)^2)^{\frac{1}{2}}\,dxdt\\
\le & \bigg((1+\delta_2^2)(\tau^{++}_{1jk})^2+ \bigg(\frac{\delta_2\nu^+_{1jk}\tau^{++}_{1jk}\tau^{+-}_{1jk}}{\tau^{++}_{1jk}+\tau^{+-}_{1jk}}\bigg)^2 +\bigg(\frac{\delta^2_2\nu^+_{1jk}\tau^{++}_{1jk}\tau^{+-}_{1jk}}{\tau^{++}_{1jk}+\tau^{+-}_{1jk}}\bigg)^2\bigg)^{\frac{1}{2}} \frac{\tau^{+-}_{1jk}}{\tau^{++}_{1jk}+\tau^{+-}_{1jk}} |\tilde{D}^+_{1jk}| \\
& + \bigg((\tau^{+-}_{1jk})^2+ \bigg(\frac{\delta_2\nu^+_{1jk}\tau^{++}_{1jk}\tau^{+-}_{1jk}}{\tau^{++}_{1jk}+\tau^{+-}_{1jk}}\bigg)^2 +\bigg(\frac{\delta^2_2\nu^+_{1jk}\tau^{++}_{1jk}\tau^{+-}_{1jk}}{\tau^{++}_{1jk}+\tau^{+-}_{1jk}}\bigg)^2\bigg)^{\frac{1}{2}} \frac{\tau^{++}_{1jk}}{\tau^{++}_{1jk}+\tau^{+-}_{1jk}} |\tilde{D}^+_{1jk}| 
\end{split}
\]
\[
\begin{split}
= & \bigg((1+\delta_2^2)(\tau^{+-}_{1jk})^2+ \bigg(\frac{\delta_2\nu^+_{1jk}(\tau^{+-}_{1jk})^2}{\tau^{++}_{1jk}+\tau^{+-}_{1jk}}\bigg)^2 +\bigg(\frac{\delta_2^2\nu^+_{1jk}(\tau^{+-}_{1jk})^2}{\tau^{++}_{1jk}+\tau^{+-}_{1jk}}\bigg)^2\bigg)^{\frac{1}{2}} \frac{\tau^{++}_{1jk}}{\tau^{++}_{1jk}+\tau^{+-}_{1jk}} |\tilde{D}^+_{1jk}| \\
& + \bigg((\tau^{+-}_{1jk})^2+ \bigg(\frac{\delta_2\nu^+_{1jk}\tau^{++}_{1jk}\tau^{+-}_{1jk}}{\tau^{++}_{1jk}+\tau^{+-}_{1jk}}\bigg)^2 +\bigg(\frac{\delta_2^2\nu^+_{1jk}\tau^{++}_{1jk}\tau^{+-}_{1jk}}{\tau^{++}_{1jk}+\tau^{+-}_{1jk}}\bigg)^2\bigg)^{\frac{1}{2}} \frac{\tau^{++}_{1jk}}{\tau^{++}_{1jk}+\tau^{+-}_{1jk}} |\tilde{D}^+_{1jk}| \\
\le & \bigg(\bigg(2S_M^2+\frac{2S_M^4}{(1-2\lambda_2')^2 S_m^2}\bigg)^{\frac{1}{2}} + \sqrt{3}S_M\bigg)(\beta_2-\beta_1) |\tilde{D}^+_{1jk}| \\
\le & \bigg(\sqrt{2}\bigg(1+\frac{64}{49}\bigg(\frac{S_M}{S_m}\bigg)^2\bigg)^{\frac{1}{2}} + \sqrt{3}\bigg)S_M(\beta_2-\beta_1) |\tilde{D}^+_{1jk}|,
\end{split}
\]
where $S_m:=\min_{[r_1,r_2]}(\omega_2-\omega_1)>0.$ Similarly, from (\ref{choice-lamda-i}), (\ref{choice-beta_i}), (\ref{proof:s-lemma-0-2-3}), (\ref{finalproof-10}), (\ref{proof:s-lemma-0-5-1}), and (iii), (v), and (vi) in Lemma \ref{lem:patching-function},
\[
\begin{split}
\int_{\tilde{D}^-_{1jk}} & |\nabla(\varphi^-_{1jk},\psi_2)|\,dxdt = \int_{T^{-,1}_{1jk}\cup T^{-,2}_{1jk}\cup T^{-,3}_{1jk}\cup T^{-,4}_{1jk}\cup R^{-,5}_{1jk}}|\nabla(\varphi^-_{1jk},\psi_2)|\,dxdt \\
= & \int_{T^{-,1}_{1jk}\cup T^{-,2}_{1jk}\cup T^{-,3}_{1jk}\cup T^{-,4}_{1jk}}((1+\delta_2^2)(\tau^{-+}_{1jk})^2+(\varphi^-_{1jk})^2+((\psi_2)_t)^2)^{\frac{1}{2}}\,dxdt\\
& + \int_{R^{-,5}_{1jk}}((\tau^{--}_{1jk})^2 +(\varphi^-_{1jk})^2+((\psi_2)_t)^2)^{\frac{1}{2}}\,dxdt\\
\le & \bigg((1+\delta_2^2)(\tau^{-+}_{1jk})^2+ \bigg(\frac{\delta_2\nu^-_{1jk}\tau^{-+}_{1jk}\tau^{--}_{1jk}}{\tau^{-+}_{1jk}+\tau^{--}_{1jk}}\bigg)^2 +\bigg(\frac{\delta_2^2\nu^-_{1jk}\tau^{-+}_{1jk}\tau^{--}_{1jk}}{\tau^{-+}_{1jk}+\tau^{--}_{1jk}}\bigg)^2\bigg)^{\frac{1}{2}} \frac{\tau^{--}_{1jk}}{\tau^{-+}_{1jk}+\tau^{--}_{1jk}} |\tilde{D}^-_{1jk}| \\
& + \bigg((\tau^{--}_{1jk})^2+ \bigg(\frac{\delta_2\nu^-_{1jk}\tau^{-+}_{1jk}\tau^{--}_{1jk}}{\tau^{-+}_{1jk}+\tau^{--}_{1jk}}\bigg)^2 +\bigg(\frac{\delta_2^2\nu^-_{1jk}\tau^{-+}_{1jk}\tau^{--}_{1jk}}{\tau^{-+}_{1jk}+\tau^{--}_{1jk}}\bigg)^2\bigg)^{\frac{1}{2}} \frac{\tau^{-+}_{1jk}}{\tau^{-+}_{1jk}+\tau^{--}_{1jk}} |\tilde{D}^-_{1jk}| \\
= & \bigg((1+\delta_2^2)(\tau^{-+}_{1jk})^2+ \bigg(\frac{\delta_2\nu^-_{1jk}\tau^{-+}_{1jk}\tau^{--}_{1jk}}{\tau^{-+}_{1jk}+\tau^{--}_{1jk}}\bigg)^2 +\bigg(\frac{\delta_2^2\nu^-_{1jk}\tau^{-+}_{1jk}\tau^{--}_{1jk}}{\tau^{-+}_{1jk}+\tau^{--}_{1jk}}\bigg)^2\bigg)^{\frac{1}{2}} \frac{\tau^{--}_{1jk}}{\tau^{-+}_{1jk}+\tau^{--}_{1jk}} |\tilde{D}^-_{1jk}| \\
& + \bigg((\tau^{-+}_{1jk})^2+ \bigg(\frac{\delta_2\nu^-_{1jk}(\tau^{-+}_{1jk})^2}{\tau^{-+}_{1jk}+\tau^{--}_{1jk}}\bigg)^2 +\bigg(\frac{\delta_2^2\nu^-_{1jk}(\tau^{-+}_{1jk})^2}{\tau^{-+}_{1jk}+\tau^{--}_{1jk}}\bigg)^2\bigg)^{\frac{1}{2}} \frac{\tau^{--}_{1jk}}{\tau^{-+}_{1jk}+\tau^{--}_{1jk}} |\tilde{D}^-_{1jk}| \\
\le & \bigg(2S_M+ \bigg(S_M^2+\frac{2S_M^4}{(1-2\lambda_2')^2 S_m^2}\bigg)^{\frac{1}{2}} \bigg)(\beta_2-\beta_1) |\tilde{D}^-_{1jk}| \\
\le & \bigg(2+\bigg(1+\frac{128}{49}\bigg(\frac{S_M}{S_m}\bigg)^2\bigg)^{\frac{1}{2}}\bigg)S_M(\beta_2-\beta_1) |\tilde{D}^-_{1jk}|.
\end{split}
\]
Also,
\[
\begin{split}
\int_{\tilde{D}_{2j}} & |\nabla(\varphi_{2j},\psi_2)|\,dxdt = \int_{T^{1}_{2j}\cup T^{2}_{2j}\cup T^{3}_{2j}\cup T^{4}_{2j}\cup R^{5}_{2j}}|\nabla(\varphi_{2j},\psi_2)|\,dxdt \\
= & \int_{T^{1}_{2j}\cup T^{2}_{2j}\cup T^{3}_{2j}\cup T^{4}_{2j}}((1+\delta_2^2)(\tau^{+}_{2j})^2+(\varphi_{2j})^2+((\psi_2)_t)^2)^{\frac{1}{2}}\,dxdt\\
& + \int_{R^{5}_{2j}}((\tau^{-}_{2j})^2 +(\varphi_{2j})^2+((\psi_2)_t)^2)^{\frac{1}{2}}\,dxdt\\
\le & \bigg((1+\delta_2^2)(\tau^{+}_{2j})^2+ \bigg(\frac{\delta_2\nu_{2j}\tau^{+}_{2j}\tau^{-}_{2j}}{\tau^{+}_{2j}+\tau^{-}_{2j}}\bigg)^2 +\bigg(\frac{\delta_2^2\nu_{2j}\tau^{+}_{2j}\tau^{-}_{2j}}{\tau^{+}_{2j}+\tau^{-}_{2j}}\bigg)^2\bigg)^{\frac{1}{2}} \frac{\tau^{-}_{2j}}{\tau^{+}_{2j}+\tau^{-}_{2j}} |\tilde{D}_{2j}| 
\end{split}
\]
\[
\begin{split}
& + \bigg((\tau^{-}_{2j})^2+ \bigg(\frac{\delta_2\nu_{2j}\tau^{+}_{2j}\tau^{-}_{2j}}{\tau^{+}_{2j}+\tau^{-}_{2j}}\bigg)^2 +\bigg(\frac{\delta_2^2\nu_{2j}\tau^{+}_{2j}\tau^{-}_{2j}}{\tau^{+}_{2j}+\tau^{-}_{2j}}\bigg)^2\bigg)^{\frac{1}{2}} \frac{\tau^{+}_{2j}}{\tau^{+}_{2j}+\tau^{-}_{2j}} |\tilde{D}_{2j}| \\
\le & \frac{2S_M\tau^{-}_{2j}}{\tau^{+}_{2j}+\tau^{-}_{2j}} |\tilde{D}_{2j}| +\frac{\sqrt{3}S_M\tau^{+}_{2j}}{\tau^{+}_{2j}+\tau^{-}_{2j}} |\tilde{D}_{2j}|\le 2S_M |\tilde{D}_{2j}|.
\end{split}
\]
Combining these estimates, we obtain that
\[
\int_{\Omega_{t_1}^{t_2}}| \nabla z_2-\nabla z_1|\,dxdt \le c_2((\beta_2 -\beta_1)|Q_1|+|Q_2|)\le c_2((\beta_2 -\beta_1)|Q|+|Q_2|),
\]
where
\[
c_2:=\frac{6S_M^2}{S_m}>0.
\]
Thus, (i) for $i=2$ holds.

In this final stage, we show that if $i_0=2$, then (j) for $i=2$ holds.  So we assume $i_0=2.$ Hence, for all $i\ge i_0=2,$
\begin{equation}\label{finalproof-11}
\zeta_0^\pm(s,r)=1\;\;\forall(s,r)\in U^\pm_{i-1}.
\end{equation}

Let $D\in\mathcal{F}_{1}.$ Then from (b), (c), and (e) for $i=1$, there are five numbers $j_1,\ldots, j_5\in\N$ with
\[
T_{1j_1}\cup T_{1j_2}\cup T_{1j_3}\cup T_{1j_4}\cup R_{1j_5} \subset D
\]
such that
\[
|D|=|T_{1j_1}\cup T_{1j_2}\cup T_{1j_3}\cup T_{1j_4}\cup R_{1j_5}| .
\]
We now observe that
\[
\begin{split}
\int_D & \zeta^+_0((v_2)_x,(w_2)_t)\,dxdt= \int_{T_{1j_1}\cup T_{1j_2}\cup T_{1j_3}\cup T_{1j_4}\cup R_{1j_5}}\zeta^+_0((v_2)_x,(w_2)_t)\,dxdt \\
& = \sum_{k\in\N} \bigg(\int_{\tilde{D}^+_{1j_1 k}}+ \int_{\tilde{D}^+_{1j_2 k}}+ \int_{\tilde{D}^+_{1j_3 k}}+ \int_{\tilde{D}^+_{1j_4 k}}+ \int_{\tilde{D}^-_{1j_5 k}} \bigg) \zeta^+_0((v_2)_x,(w_2)_t)\,dxdt.
\end{split}
\]
Let $k\in\N.$ Then from (\ref{proof:s-lemma-0-2-0}), (\ref{finalproof-11}), the first of (g) for $i=2$, and (vi) in Lemma \ref{lem:patching-function},
\[
\int_{\tilde{D}^+_{1j k}} \zeta^+_0((v_2)_x,(w_2)_t)\,dxdt = \frac{\tau^{+-}_{1j k}}{\tau^{++}_{1j k} + \tau^{+-}_{1j k}} |\tilde{D}^+_{1j k}|\ge \frac{\tau^{++}_{1j k}}{\tau^{++}_{1j k} + \tau^{+-}_{1j k}} |\tilde{D}^+_{1j k}|
\]
for $j=j_1,\ldots,j_4$, and
\[
\int_{\tilde{D}^-_{1j_5 k}} \zeta^+_0((v_2)_x,(w_2)_t)\,dxdt = \frac{\tau^{--}_{1j_5 k}}{\tau^{-+}_{1j_5 k} + \tau^{--}_{1j_5 k}} |\tilde{D}^-_{1j k}|.
\]
Thus, from (\ref{choice-alpha-i}), (\ref{choice-beta_i}), (\ref{finalproof-12}), and (\ref{finalproof-13}),
\[
\int_D  \zeta^+_0((v_2)_x,(w_2)_t)\,dxdt \ge \frac{\lambda_2-\lambda_2'}{1-2\lambda_2'}|D|\ge c_3(\beta_2-\beta_1)|D|,
\]
where $c_3:=\frac{\kappa_0\ell_0}{36}>0$. Likewise, one can check that
\[
\int_D  \zeta^-_0((v_2)_x,(w_2)_t)\,dxdt \ge c_3(\beta_2-\beta_1)|D|.
\]
Hence, the first inequalities in (j) for $i=2$ hold.

Next, for $k\in\N$ and $j=j_1,\ldots,j_4$, it follows from (\ref{finalproof-14}) and (\ref{finalproof-15}) that
\[
\int_{\tilde{D}^+_{1j k}} \zeta^+_0((v_2)_x,(w_2)_t)\,dxdt = \frac{\tau^{+-}_{1j k}}{\tau^{++}_{1j k} + \tau^{+-}_{1j k}} |\tilde{D}^+_{1j k}| \ge (1-(\beta_2-\beta_1))|\tilde{D}^+_{1j k}|;
\]
thus,
\[
\int_D \zeta^+_0((v_2)_x,(w_2)_t)\,dxdt\ge (1-(\beta_2-\beta_1))|T_{1j_1}\cup T_{1j_2}\cup T_{1j_3}\cup T_{1j_4}|.
\]
On the other hand, from (\ref{finalproof-11}) and the first of (g) for $i=1$,
\[
\int_D \zeta^+_0((v_1)_x,(w_1)_t)\,dxdt=|T_{1j_1}\cup T_{1j_2}\cup T_{1j_3}\cup T_{1j_4}|
\]
so that
\[
\int_D \zeta^+_0((v_2)_x,(w_2)_t)\,dxdt\ge (1-(\beta_2-\beta_1))\int_D \zeta^+_0((v_1)_x,(w_1)_t)\,dxdt.
\]
Similarly, one can check that
\[
\int_D \zeta^-_0((v_2)_x,(w_2)_t)\,dxdt\ge (1-(\beta_2-\beta_1))\int_D \zeta^-_0((v_1)_x,(w_1)_t)\,dxdt.
\]
Hence, the second inequalities in (j) for $i=2$ hold.

The second step is now complete.

\subsection*{Acknowledgments} The authors would like to thank the anonymous referees for valuable suggestions that greatly improved the presentation of the paper.
H. J. Choi was supported by the National Research Foundation of Korea (grant RS-2023-00280065). S. Kim was supported by the National Research Foundation of Korea (grant NRF-2022R1F1A1063379, RS-2023-00217116) and Korea Institute for Advanced Study(KIAS) grant funded by the Korea government (MSIP). Y. Koh was supported by the National Research Foundation of Korea (grant NRF-2022R1F1A1061968) and Korea Institute for Advanced Study(KIAS) grant funded by the Korea government (MSIP).

\subsubsection*{\emph{\textbf{Ethical Statement.}}} The manuscript has not been submitted to more than one journal for simultaneous consideration.  The manuscript has not been published previously.


\subsubsection*{\emph{\textbf{Conflict of Interest:}}}  The authors declare that they have no conflict of interest.

\end{document}